\documentclass[11pt]{amsart}
\usepackage{amsfonts}
\usepackage{amsmath}
\usepackage{amssymb}
\usepackage{amsthm}
\usepackage{cases}
\usepackage{color}
\usepackage{enumerate}
\usepackage{float}
\usepackage{graphicx}
\usepackage{geometry}
\usepackage{hyperref}
\usepackage{mathrsfs}
\usepackage{multirow}
\usepackage[numbers,sort&compress]{natbib}
\usepackage{subeqnarray}

\allowdisplaybreaks
\numberwithin{equation}{section}

\theoremstyle{plain}
\newtheorem{prop}{Proposition}[section]

\newtheorem{lemm}[prop]{Lemma}

\newtheorem{theorem}[prop]{Theorem}
\newtheorem*{thrmA}{Theorem A}

\theoremstyle{definition}
\newtheorem{defi}[prop]{Definition}

\newtheorem{exam}[prop]{Example}
\newtheorem{rema}[prop]{Remark}

\renewcommand\aa{a}

\newcommand\bb{b}

\newcommand\cc{c}
\newcommand{\chart}{\mathsf{char}}
\newcommand\dd{d}

\newcommand\ff{f}
\newcommand\ffb{\lm{\ff}} 
\renewcommand\geq{\geqslant}
\renewcommand\gg{g}

\newcommand{\gsb}{Gr\"{o}bner--Shirshov basis}
\newcommand\hh{h}
\newcommand\HS[1]{\leavevmode\null\hspace{#1mm}}
\newcommand\id{\mathsf{Id}}
\newcommand\Id[1]{\id(#1)}

\newcommand\ii{i}
\newcommand\Irr[1]{\mathsf{Irr}(#1)}
\newcounter{ITEM}
\newcommand\ITEM[1]{\setcounter{ITEM}{#1}\leavevmode\hbox{\rm(\roman{ITEM})}}
\newcommand\jj{j}
\newcommand\kk{k}

\newcommand\lcoe[1]{\mathsf{lc}(#1)}

\renewcommand\leq{\leqslant}

\newcommand\lm[1]{\overline{#1}}
\newcommand\lnormed[2]{[#1\wdots#2]{_{_\mathsf{L}}}}
\newcommand\Lnormed[1]{[#1]{_{_\mathsf{L}}}}
\newcommand\Lnormeda[3]{[#1,\ova{#2]{_{_\mathsf{L}}}\cdot #3}}

\newcommand\mA{\mathcal{A}}
\newcommand\mB{\mathcal{B}}
\newcommand\mm{m}

\newcommand\MP[1]{\mathcal{MP}_{\kk}(#1)}
\newcommand\mpb{\mathsf{MPB}_{\kk}(\XX)}
\newcommand\mV{\mathcal{V}}
\newcommand\mW{\mathcal{W}}
\newcommand\nrp{normal~$\RR$-polynomial}

\newcommand\nsp{normal~$\SS$-polynomial}
\newcommand\nsps{normal~$\SS$-polynomials}
\newcommand\nn{n}

\newcommand\ov[1]{\overline{#1}}
\newcommand\ova[1]{\overrightarrow{#1}}
\newcommand\pb{\mathsf{Pb}}
\newcommand\pdots{\mathrel{\HS{0.2}{\cdot}{\cdot}{\cdot}\HS{0.2}}}
\newcommand\PMP[2]{\mathcal{MP}_{\kk}(#1\,\vert\,#2)}
\newcommand\pp{p}

\newcommand\qq{q}
\newcommand\quldots{ ...}
 
 \newcommand\RR{R}

\renewcommand\ss{s}
\renewcommand\SS{S}
\renewcommand\tt{t}
\newcommand\tl[1]{\widetilde{#1}}

\newcommand\wdots{, ...\HS{0.2}, }

\newcommand{\wt}{\mathsf{wt}}

\newcommand\WW{W}
\newcommand\xx{x}
\newcommand\XX{X}

\newcommand\YYY{Y}
\newcommand\yy{y}
\newcommand\zz{z}

\title{Word problem for finitely presented metabelian Poisson algebras
}

\author{Zerui Zhang}
\address{Z.Z., School of Mathematical Sciences, South China Normal University, Guangzhou 510631, P. R. China}
\email{\small 295841340@qq.com}

\author{Yuqun Chen$^{\sharp}$}
\address{Y.C., School of Mathematical Sciences, South China Normal University, Guangzhou 510631, P. R. China}
\email{yqchen@scnu.edu.cn}

\author{Leonid A. Bokut$^{\dagger}$}
\address{L.A.B., School of Mathematical Sciences, South China Normal University Guangzhou 510631, P. R. China; Sobolev Institute of mathematics, Novosibirsk, 630090, Russia; Novosibirsk State University, Novosibirsk 630090, Russia}
\email{bokut@math.nsc.ru}

\thanks{${}^{\ddag}$Supported by the NNSF of China (11571121), the NSF of Guangdong Province (2017A030313002) and the Science and Technology Program of Guangzhou (201707010137)}

\thanks{${}^{\sharp}$ Corresponding author}

\thanks{${}^{\dagger}$ Supported by Russian Science Foundation (project 14-21-00065)}


\begin{document}

\begin{abstract}
  We first construct a linear basis for a free metabelian Poisson algebra generated by an arbitrary well-ordered set. It turns out that such a linear basis  depends on the characteristic of the underlying field.
 Then we elaborate the method of Gr\"{o}bner--Shirshov bases for metabelian Poisson algebras. Finally, we show that the word problem for finitely presented metabelian Poisson algebras are solvable.
\end{abstract}

\maketitle

\section{Introduction}\label{Intro}

We recall that a \emph{Poisson algebra} is a vector space~$\mathcal{P}$ over a field~$k$ endowed with two bilinear operations, a multiplication written~$\cdot$ and a Poisson bracket written~$(-,-)$, such that~$(\mathcal{P}, \cdot)$ is an associative commutative algebra, $(\mathcal{P}, (-,-))$ is a Lie algebra, and~$\mathcal{P}$ satisfies the \emph{Leibniz identity}
\begin{equation}\label{leibniz}
(\xx, \yy \cdot \zz) = (\xx, \yy) \cdot \zz + \yy \cdot (\xx, \zz).
\end{equation}
 A Poisson algebra~$\mathcal{P}$ is called \emph{abelian} if~$\mathcal{P}$ is just a vector space with trivial product, that is, $\xx\cdot\yy=(\xx,\yy)=0$ for all~$\xx$ and~$\yy$ in~$\mathcal{P}$.  A Poisson algebra~$\mathcal{P}$  is called \emph{metabelian} if~$\mathcal{P}$ is an extension of an abelian Poisson algebra by another abelian Poisson algebra~\cite{metabelian-poisson}. It is clear that metabelian Poisson algebras form a variety, and  therefore, free metabelian Poisson algebras exist.

It is known that the word problem for finitely presented metabelian Lie algebras are solvable, for instance, see~\cite{hall,metabelian-word-problem,word-problem}. By applying Gr\"{o}bner-Shirshov bases theory~\cite{Buchberger65} and the Hilbert Basis Theorem~\cite{hilbert-basis-theorem}, we know that the word problem for finitely presented commutative algebras are also solvable. In our previous paper~\cite{poisson-cd-lemma}, we proved that the word problem for Poisson algebras in several nontrivial special cases are solvable, and, in this paper, we continue to deal with the word problem for finitely presented metabelian Poisson algebras.  The main result we prove below is as follows:
\begin{thrmA}
 The word problem for an arbitrary finitely presented metabelian Poisson algebra is solvable.
\end{thrmA}
 Note that when we consider the word problem or other algorithmic problems of an algebra,  we always assume that the underlying field~$\kk$ is computable. Intuitively, a field~$\kk$ is \emph{computable} if there is an algorithm that, upon input of~$\xx,\yy$ in~$\kk$ computes~$\xx+\yy,\xx\yy,-\xx,$ and computes~$\xx^{-1}$ if~$\xx$ is nonzero~\cite{gsb-book}.

  Our method of proof consists of using Gr\"obner--Shirshov bases, as introduced by Buchberger~\cite{Buchberger65} and Shirshov~\cite{shirshov1962lie-cd-lemma}.
 The principle of our construction is directly reminiscent of what was done for metabelian Lie algebras~\cite{met-lie-cd-lemma}. However, the extension is not obvious. One of the difficulties is that, the reasoning for metabelian Poisson algebras depends on the characteristic of the underlying field. Another difficulty is that, the number of the bilinear operations for a Poisson algebra is more than that of a Lie algebra. Both of these two difficulties call for lots of discussions, and so, new ideas is in urgent. Therefore, we try to avoid discussions as many as possible, at the cost of possibly calculating more compositions for some concrete examples.

 The paper is organized as follows: In Section~\ref{free-met-poisson}, we construct linear bases for free metabelian Poisson algebras over fields of different characteristics. In Section~\ref{main-result}, we first elaborate a Gr\"{o}bner--Shirshov basis method for metabelian Poisson algebras and then establish our main results about the solvability of the word problem for finitely presented metabelian Poisson algebras.

\section{On free metabelian Poisson algebras}\label{free-met-poisson}
Our aim in this section is to construct linear bases for free metabelian Poisson algebras over fields of different characteristics. Denote by~$\MP\XX$ the free metabelian Poisson algebras (freely) generated by a well-ordered set~$\XX$ over a field~$\kk$, and by~$\chart(\kk)$ the characteristic of the underlying field~$\kk$. (We use~$\kk$ as a subscript because the linear basis of~$\MP\XX$ depends on~$\chart(\kk)$.)
\subsection{A linear generating set for~$\MP\XX$}
The notion of a metabelian Poisson algebra was introduced in~\cite{metabelian-poisson}. Here, we shall first state the definition in a slightly different form, which is equivalent to the original one. Then we shall construct a linear basis for the free metabelian Poisson algebra~$\MP\XX$. Before going there, we first introduce several needed notations below. For all subspaces~$\mV$ and~$\mW$ of a Poisson algebra~$\mathcal{P}$, denote by~$\mV+\mW$ the sum of the vector spaces~$\mV$ and~$\mW$, and we define
$$
(\mV,\mW)=\mathsf{span}_\kk\{(\xx,\yy)\in \mathcal{P}\mid \xx\in \mV, \yy\in \mW\}
$$
and
$$
\mV\cdot\mW=\mathsf{span}_\kk\{\xx\cdot\yy\in \mathcal{P}\mid \xx\in \mV, \yy\in \mW\}.
$$

\begin{lemm}\label{meta-equiv-def}
 A Poisson algebra~$\mathcal{P}$ is metabelian if and only if the equality
\begin{equation}\label{metabelian-poisson-def}
\mathcal{P}^2\cdot \mathcal{P}^2+(\mathcal{P}^2, \mathcal{P}^2)=0
\end{equation}
holds, where~$\mathcal{P}^2$ is defined to be the vector space~$\mathcal{P}\cdot \mathcal{P}+(\mathcal{P},\mathcal{P})$.
\end{lemm}
\begin{proof}
  Assume that~$0 \longrightarrow \mA \longrightarrow \mathcal{P} \longrightarrow \mB \longrightarrow 0$
  is an exact sequence of Poisson algebras, where~$\mA$ and~$\mB$ are abelian. Since~$\mB$ is an abelian Poisson algebra, we deduce~$\mathcal{P}^2+\mA=\mA$ and thus we have~$\mathcal{P}^2\subseteq\mA$. Since~$\mA$ is also abelian, we obtain
  $$
  \mathcal{P}^2\cdot \mathcal{P}^2+(\mathcal{P}^2, \mathcal{P}^2)=0.
  $$
  Conversely, assume that~\eqref{metabelian-poisson-def} holds. Then both~$\mathcal{P}^2$ and~$\mathcal{P}{/}\mathcal{P}^2$ are abelian. Moreover, there is an obvious exact sequence~$0 \longrightarrow \mathcal{P}^2 \longrightarrow \mathcal{P} \longrightarrow \mathcal{P}{/}\mathcal{P}^2\longrightarrow 0$.
\end{proof}
 As a result, for all~$\xx, \yy,\zz,\zz'$ in a metabelian Poisson algebra~$\mathcal{P}$, we have
\begin{equation}\label{length-4}
((\xx,\yy),(\zz,\zz'))=((\xx,\yy),\zz\cdot \zz')= (\xx,\yy)\cdot(\zz,\zz')=\xx \cdot \yy\cdot (\zz,\zz')= \xx\cdot\yy\cdot \zz\cdot \zz'=0.
\end{equation}
In particular, since~$((\xx,\yy),(\zz,\zz'))=(((\xx,\yy),\zz ),\zz')-(((\xx,\yy),\zz' ),\zz)$, by~$\eqref{length-4}$ we deduce
\begin{equation}\label{lie-commu}
  (((\xx,\yy),\zz ),\zz')= (((\xx,\yy),\zz' ),\zz).
\end{equation}
Similarly, we also have
\begin{equation}\label{cdot-anti-commu}
  ((\xx,\yy),\zz )\cdot \zz'= -((\xx,\yy),\zz' )\cdot \zz
\end{equation}
For all elements~$\xx_1\wdots\xx_\nn $ of a metabelian Poisson algebra, to make the notation shorter, we define
\begin{equation*}
\lnormed{\xx_1}{\xx_\nn}=(\quldots ((\xx_1, \xx_2) ,\xx_3) ,\quldots , \xx_\nn) \ \ \mbox{(left-normed bracketing).}
\end{equation*}
In particular, we have~$\Lnormed{\xx_1,\xx_2}=(\xx_1,\xx_2)$ and~$\Lnormed{\xx_1}=\xx_1$.

In the sequel, we will frequently reorder a sequence. For all~$\aa_1\wdots\aa_n$ in~$\XX$, for every~$\ff$ in~$\MP\XX$, assume that~$\dd_1\wdots\dd_n$ is a reordering of~$\aa_1\wdots\aa_n$ satisfying~$\dd_1\leq \pdots\leq \dd_n$. Then we define
$$
\ova{\aa_1\wdots\aa_n}=\dd_1\wdots\dd_n, \ \ova{\aa_1\cdot\pdots\cdot\aa_n}=\dd_1\cdot \pdots \cdot\dd_n,\ \
\Lnormed{\ff,\ova{\aa_{1}\wdots\aa_n}}
=\Lnormed{\ff,\dd_{1}\wdots\dd_{n}}
$$
and
$$
 \Lnormeda{\ff}{\aa_{1}\wdots\aa_{n-1}}{\aa_n}
=\Lnormed{\ff,\dd_{1}\wdots\dd_{n-1}}\cdot \dd_n.
$$
For simplicity, we sometimes omit the multiplication~$\cdot$, for instance, we can denote~$\aa_1\cdot\pdots\cdot\aa_n$ by~$\aa_1...\aa_n$.

 We are now ready to construct a linear generating set~$\mpb$ for~$\MP\XX$, and we shall prove that~$\mpb$ is in fact a linear basis for~$\MP\XX$ in Theorem~\ref{iso}.
\begin{lemm}\label{generating-set}
    Let~$\XX$ be a well-ordered set, and define a series of sets~$\YYY_1 \wdots \YYY_5$ as follows:

    $\YYY_1=\{\lnormed{\aa_1}{\aa_n}\mid \nn\geq 1, \aa_1\wdots\aa_n \in \XX, \aa_2\leq \pdots\leq \aa_n,   \aa_1
    >\aa_2~\mbox{if}~\nn\geq 2\}$,

    $\YYY_2=\{\aa_1...\aa_t\mid \aa_1\wdots\aa_t\in \XX, \aa_1\leq \pdots\leq \aa_t,2\leq \tt\leq 3 \}$,

    $\YYY_3=\{ [\aa_1,\aa_2]_{_\mathsf{L}}\cdot \aa_3\mid \aa_1,\aa_2,\aa_3\in \XX,\aa_1 >\aa_2 \}$,

    $\YYY_4=\{ [\aa_1,\aa_2,\aa_3]_{_\mathsf{L}}\cdot \aa_4\mid \aa_1\wdots\aa_4\in \XX, \aa_2<\aa_1\leq \aa_4,  \aa_2\leq \aa_3<\aa_4\}$,

     $\YYY_5=\{\lnormed{\aa_1}{\aa_n}\cdot \aa_{n+1}\mid \aa_1\wdots\aa_{n+1} \in \XX, n\geq 3, \aa_1>\aa_2, \aa_2\leq \pdots\leq \aa_{n+1}\}$.\\
    Define~$$\mpb:=
    \begin{cases}\YYY_1\cup\YYY_2\cup\YYY_3\cup \YYY_4,\ \  \mbox{ if }~\chart(\kk)\neq 2,\\
    \YYY_1\cup\YYY_2\cup\YYY_3\cup \YYY_5,\ \  \mbox{ if }~\chart(\kk)= 2.
    \end{cases}$$
     Then~$\mpb$ is a linear generating set for~$\MP\XX$.
\end{lemm}
\begin{proof}
Note that a metabelian Poisson algebra is also a metabelian Lie algebra, where a metabelian Lie algebra is by definition an extension of an abelian Lie algebra by another abelian Lie algebra. It is known that~$\YYY_1$ is a linear basis of the free metabelian Lie algebra generated by a well-ordered set~$\XX$~\cite{bokut-basis-for-metabelian-lie}.
In particular, every monomial in~$\MP\XX$ that involves only the Poisson brackets can be written as a linear combination of elements in~$\YYY_1$. On the other hand, monomials that do not involve Poisson brackets can be written as linear combinations of elements in~$\YYY_2\cup\XX$. What remains is the monomials involving both the Poisson brackets and the multiplications.

By the Leibniz identity~\eqref{leibniz}, every element of~$\MP\XX$ can be written as a linear combination of elements of the form~$\xx_1... \xx_n$, where~$\xx_1\wdots\xx_n$ lies in the metabelian Lie-subalgebra of~$\MP\XX$ generated by~$\XX$. In particular, we may assume that each~$\xx_i$ lies in~$\YYY_1$.
By identity~$\eqref{length-4}$, we have~$\xx_1... \xx_n=0$ if~$n\geq 4$, or, if~$\nn=3$ and at least one of~$\xx_1\wdots\xx_3$ does not lie in~$\XX$, or, if~$\nn=2$ and neither~$\xx_1$ nor~$\xx_2$ lies in~$\XX$. Therefore, if~$\xx_1...\xx_n$ does not belong to~$\YYY_1\cup\YYY_2\cup\YYY_3$, then we may assume
$$
\xx_1...\xx_n
=\xx_1\cdot \xx_2=\lnormed{\aa_1}{\aa_m}\cdot \aa_{m+1}
$$
 for some letters~$\aa_1\wdots\aa_{m+1} \in \XX$, $m\geq 3$ satisfying~$\aa_1>\aa_2$ and~$\aa_2\leq \pdots\leq \aa_{m}$.

We first show that~$\lnormed{\aa_1}{\aa_m}\cdot \aa_{m+1}$ can be written as a linear combination of elements in~$\YYY_5$. If~$\aa_2\leq \aa_{m+1}$, then by~\eqref{lie-commu} and~\eqref{cdot-anti-commu}, we deduce
$$
 \lnormed{\aa_1}{\aa_m}\cdot\aa_{m+1}
=\alpha\Lnormeda{\aa_1,\aa_2}{\aa_3\wdots\aa_m}{\aa_{m+1}}
=\alpha\Lnormeda{\aa_1}{\aa_2,\aa_3\wdots\aa_m}{\aa_{m+1}}
$$
for some integer~$\alpha$ equals to~$1$ or~$-1$. If~$\aa_{m+1}<\aa_2$, then by~\eqref{lie-commu} and~\eqref{cdot-anti-commu} again, we obtain
\begin{multline*}
\Lnormed{\aa_1\wdots\aa_{m}}\cdot\aa_{m+1}
=-\Lnormed{\aa_1, \aa_2, \aa_{m+1}, \aa_3\wdots\aa_{m-1}}\cdot\aa_{m}\\
=-\Lnormed{\aa_1,\aa_{m+1},\aa_2\wdots\aa_{m-1}}\cdot \aa_m
+\alpha\Lnormeda{\aa_2,\aa_{m+1}}{\aa_3\wdots\aa_m}{\aa_1}
\end{multline*}
for some integer~$\alpha$ equals to~$1$ or~$-1$, where for~$\mm=3$, the sequence~$ \aa_3\wdots\aa_{m-1}$ is empty,  the sequence~$\aa_2\wdots\aa_{m-1}$ means~$\aa_2$ and the sequence~$\aa_3\wdots\aa_{\mm}$ means~$\aa_3$. In particular, we conclude that, if~$\chart(\kk)=2$, then~$\mpb$ is a linear generating set for~$\MP\XX$.

Now we assume~$\chart(\kk)\neq 2$ and~$\mm=3$. We shall show that every~$((\aa_1,\aa_2),\aa_3)\cdot \aa_4$ in~$\YYY_5$ can be written as a linear combination of elements in~$\YYY_4$.  If~$\aa_3=\aa_4$, then by~\eqref{cdot-anti-commu} and by the fact that~$\chart(\kk)\neq 2$, we deduce~$((\aa_1,\aa_2),\aa_3)\cdot \aa_4=0$.  Therefore, we may assume~$\aa_3<\aa_4$.  If we also have~$\aa_1\leq \aa_4$, then~$((\aa_1,\aa_2),\aa_3)\cdot \aa_4$ lies in~$\YYY_4$. On the other hand, if~$\aa_4<\aa_1$, then we obtain
\begin{align*}
&2((\aa_1,\aa_2),\aa_3)\cdot \aa_4
  =((\aa_1,\aa_2),\aa_3)\cdot \aa_4-((\aa_1,\aa_2),\aa_4)\cdot \aa_3&\\
  =&((\aa_1,\aa_3),\aa_2)\cdot \aa_4 -((\aa_2,\aa_3),\aa_1)\cdot \aa_4
  -((\aa_1,\aa_4),\aa_2)\cdot \aa_3 +((\aa_2,\aa_4),\aa_1)\cdot \aa_3&\\
  =&-((\aa_1,\aa_3),\aa_4)\cdot \aa_2 +((\aa_2,\aa_3),\aa_4)\cdot \aa_1
  +((\aa_1,\aa_4),\aa_3)\cdot \aa_2 -((\aa_2,\aa_4),\aa_3)\cdot \aa_1&\\
  =&-((\aa_1,\aa_3),\aa_4)\cdot \aa_2 +((\aa_1,\aa_4),\aa_3)\cdot \aa_2-((\aa_3,\aa_2),\aa_4)\cdot \aa_1 +((\aa_4,\aa_2),\aa_3)\cdot \aa_1&\\
   =&-((\aa_4,\aa_3),\aa_1)\cdot \aa_2 +((\aa_4,\aa_3),\aa_2)\cdot \aa_1=2((\aa_4,\aa_3),\aa_2)\cdot \aa_1&\\
   =&2((\aa_4,\aa_2),\aa_3)\cdot \aa_1-2((\aa_3,\aa_2),\aa_4)\cdot \aa_1.&
\end{align*}
Since~$\chart(\kk)\neq 2$, we deduce
$$
((\aa_1,\aa_2),\aa_3)\cdot \aa_4=((\aa_4,\aa_2),\aa_3)\cdot \aa_1-((\aa_3,\aa_2),\aa_4)\cdot \aa_1,$$
where~$((\aa_3,\aa_2),\aa_4)\cdot \aa_1$ lies in~$\YYY_4$ if~$\aa_3>\aa_2$ and~$((\aa_3,\aa_2),\aa_4)\cdot \aa_1$ is~$0$ if~$\aa_3=\aa_2$.
Therefore,  we finally obtain a linear combination of elements in~$\YYY_4$ for~$((\aa_1,\aa_2),\aa_3)\cdot \aa_4$.

Finally, assume~$\chart(\kk)\neq 2$ and~$\mm\geq 4$, then by~$\eqref{lie-commu}$ and~$\eqref{cdot-anti-commu}$, we have
\begin{align*}
& \lnormed{\aa_1}{\aa_m}\cdot \aa_{m+1}
 =-((\lnormed{\aa_1}{\aa_{m-2}},\aa_{m-1}),\aa_{m+1})\cdot \aa_{m}&\\
 =&-((\lnormed{\aa_1}{\aa_{m-2}},\aa_{m+1}),\aa_{m-1})\cdot \aa_{m}
 =((\lnormed{\aa_1}{\aa_{m-2}},\aa_{m+1}),\aa_{m})\cdot \aa_{m-1}&\\
 =&-((\lnormed{\aa_1}{\aa_{m-2}},\aa_{m}),\aa_{m-1})\cdot \aa_{m+1}
 =-\lnormed{\aa_1}{\aa_m}\cdot \aa_{m+1}.&
\end{align*}
Since~$\chart(\kk)\neq 2$, we deduce~$\lnormed{\aa_1}{\aa_m}\cdot \aa_{m+1}=0$ if~$\mm\geq 4$.
\end{proof}

\subsection{The linear independence of~$\mpb$}
With the notations of Lemma~\ref{generating-set}, to show that~$\mpb$ is a linear basis for~$\MP\XX$, the strategy is to construct another metabelian Poisson algebra that is isomorphic to~$\MP\XX$. (We can also use the method of Gr\"{o}bner--Shirshov bases for Poisson algebras~\cite{poisson-cd-lemma} to prove the linear independence of~$\mpb$, but in order to make the paper self-contained, we choose the first way.) Let~$\kk\mpb$ be the linear space over~$\kk$ with a linear basis~$\mpb$. We shall define bilinear operations~$\circ$ and~$\{-,-\}$ on~$\kk\mpb$ such that~$(\kk\mpb,\circ, \{-,-\})$ becomes a metabelian Poisson algebra generated by~$\XX$.

For every~$\WW$ in~$\mpb$, define the \emph{length}~$\ell(\WW)$ of~$\WW$ to be the number of letters (with repetitions) that appear in~$\WW$. For instance, for all~$\aa_1\wdots\aa_n$ in~$\XX$, we have~${\ell(\lnormed{\aa_1}{\aa_n})=n}$. In the following definition, for every~$\WW$ in~$\mpb$, the notation~$0\cdot \WW$ means~$0$.
\begin{defi}\label{defi-prod}
We define bilinear operations~$\circ$ and~$\{-,-\}$ on~$\kk\mpb$ as follows: For every~$\aa$ in~$\XX$, for all~$\WW_1, \WW_2$ in~$\mpb$,

\ITEM1 If~$\ell(\WW_1)\geq 2$ and~$\ell(\WW_2)\geq 2$, then define
$$\WW_1\circ \WW_2=\{\WW_1,\WW_2\}=0.$$

\ITEM2 For~$\bb$ in~$\XX$,  we define
$$
\bb\circ \aa=\ova{\bb\cdot \aa},$$ and define
$$
\{\bb, \aa\}=
\begin{cases}
[\bb,\aa]_{_\mathsf{L}}, &   \mbox{if } \bb>\aa, \\
0,   & \mbox{if } \bb=\aa,  \\
-[\aa,\bb]_{_\mathsf{L}}, &   \mbox{if } \bb<\aa.
\end{cases}
$$

\ITEM3 If~$\WW_1=\aa_1\cdot \aa_2$ lies in~$\YYY_2$, then define
$$(\aa_1\cdot \aa_2)\circ \aa= \ova{\aa_1\cdot \aa_2\cdot\aa},$$ and define
$$\{\aa_1\cdot \aa_2 , \aa\}=\{\aa_1,\aa\} \cdot \aa_2+ \{\aa_2,\aa\} \cdot \aa_1,$$
where for all letters~$\aa$ and~$\bb$ in~$\XX$,  the notation~$\{\aa,\bb\}$ is already defined in~\ITEM2, and~$0\cdot \aa$ means~0.

\ITEM4 If~$\WW_1=[\aa_1,\aa_2]_{_\mathsf{L}}$ lies in~$\YYY_1$, then~we define
$$[\aa_1,\aa_2]_{_\mathsf{L}}\circ \aa=[\aa_1,\aa_2]_{_\mathsf{L}}\cdot \aa,$$ and
$$
\{[\aa_1,\aa_2]_{_\mathsf{L}}, \aa\}=
\begin{cases}
\Lnormed{\aa_1,\aa_2,\aa}, &   \mbox{if } \aa_2\leq \aa, \\
\Lnormed{\aa_1,\aa,\aa_2}-\Lnormed{\aa_2,\aa,\aa_1},   & \mbox{if } \aa_2> \aa.
\end{cases}
$$

\ITEM5 If~$\WW_1=\aa_1\cdot \aa_2\cdot \aa_3$ lies in~$\YYY_2$, then we define
$$
(\aa_1\cdot \aa_2\cdot \aa_3)\circ \aa=\{\aa_1\cdot \aa_2\cdot \aa_3,\aa\}=0.
$$

\ITEM6 If~$\WW_1=\Lnormed{\aa_1,\aa_2,\aa_3}$ lies in~$\YYY_1$, then we have~$\aa_1>\aa_2$ and~$\aa_2\leq \aa_3$. We define
$$
\{\Lnormed{\aa_1,\aa_2,\aa_3}, \aa\}=
\begin{cases}
\Lnormed{\aa_1,\aa_2,\ova{\aa_3, \aa}}, &   \mbox{if } \aa\geq \aa_2, \\
\Lnormed{\aa_1,\aa,\aa_2,\aa_3}- \Lnormed{\aa_2,\aa,\ova{\aa_1, \aa_3}}, &   \mbox{if } \aa< \aa_2.
\end{cases}
$$
The definition of~$\WW_1\circ \aa$ depends on~$\chart(\kk)$. For the case that~$\chart(\kk)\neq 2$, we define~$\WW_1\circ \aa$ as follows (the idea is to ``move" the maximal letter involved outside the Poisson brackets):
\begin{numcases}{\Lnormed{\aa_1,\aa_2,\aa_3}\circ \aa=}
0,    \ \ \ \ \  \ \ \ \ \ \ \ \ \ \ \ \ \ \ \ \ \  \mbox{if } \aa_3=\aa, \label{a3=a} \\
\Lnormed{\aa_1,\aa_2,\aa_3} \cdot \aa,    \ \  \ \ \ \mbox{if }   \aa>\aa_3, \ \aa \geq  \aa_1,  \label{amax}\\
\{\{\aa,\aa_3\},\aa_2\}\cdot \aa_1,   \ \     \mbox{if } \aa>\aa_3, \ \aa<\aa_1, \label{a1>a>a3}\\
\{\{\aa,\aa_3\},\aa_2\}\cdot \aa_1,   \ \    \mbox{if }\aa_3 > \aa, \  \aa_3<\aa_1,\label{a1>a3>a}\\
-\{ \{\aa_1,\aa_2\},\aa\} \cdot \aa_3,     \mbox{if } \aa_3 > \aa, \  \aa_3>\aa_1,\label{a3max}\\
-\Lnormed{\aa_3,\aa_2,\aa} \cdot \aa_1,
  \ \ \ \mbox{if }\aa_3=\aa_1>\aa\geq \aa_2,\label{a3=a1>aa2}\\
-\Lnormed{\aa_3,\aa,\aa_2} \cdot \aa_1,
  \ \ \ \mbox{if }\aa_3=\aa_1>\aa_2>\aa,\label{a3=a1>a2>a}
\end{numcases}
where by convention,  polynomial~$(\sum_{i}\alpha_i\Lnormed{\bb_{i,1},\bb_{i,2},\bb_{i,3}})\cdot\bb_4$ means~$\sum_{i}\alpha_i(\Lnormed{\bb_{i,1},\bb_{i,2},\bb_{i,3}}\cdot\bb_4)$ if each~$\bb_{i,j}$ lies in~$\XX$. For the case when~$\chart(\kk)=2$, we define
$$
\Lnormed{\aa_1,\aa_2,\aa_3}\circ \aa=
\begin{cases}
\Lnormeda{\aa_1,\aa_2}{\aa_3} { \aa}, &   \mbox{if } \aa\geq \aa_2, \\
\Lnormed{\aa_1,\aa,\aa_2} \cdot \aa_3- \Lnormeda{\aa_2,\aa}{\aa_1}{\aa_3}, &   \mbox{if } \aa< \aa_2.
\end{cases}
$$

\ITEM7 If~$\WW_1=[\aa_1,\aa_2]_{_\mathsf{L}}\cdot \aa_3$ lies in~$\YYY_3$, then we define
$$
([\aa_1,\aa_2]_{_\mathsf{L}}\cdot \aa_3)\circ \aa=0.
$$
The product~$\{\WW_1,\aa\}$ also depends on~$\chart(\kk)$. Thanks to~\ITEM6, we define
$$\{[\aa_1,\aa_2]_{_\mathsf{L}}\cdot \aa_3,\aa\}=\{[\aa_1,\aa_2]_{_\mathsf{L}},\aa\}\circ \aa_3.$$

\ITEM8 If~$\WW_1=\lnormed{\aa_1}{\aa_n}$ lies in~$\YYY_1$ satisfying~$\nn\geq 4$, then we define
$$
\{\lnormed{\aa_1}{\aa_n}, \aa\}=
\begin{cases}
\Lnormed{\aa_1,\aa_2,\ova{\aa_3\wdots\aa_{n},\aa}}, &   \mbox{if } \aa\geq \aa_2, \\
\lnormed{\aa_1,\aa,\aa_2}{\aa_n}- \Lnormed{\aa_2,\aa,\ova{\aa_1,\aa_3\wdots\aa_n}}, &   \mbox{if } \aa< \aa_2.
\end{cases}
$$
 Again, the product~$\lnormed{\aa_1}{\aa_n}\circ \aa$ depends on~$\chart(\kk)$. If~$\chart(\kk)\neq 2$, then we define
 $$
 \lnormed{\aa_1}{\aa_n}\circ \aa=0.
 $$
 For the case that~$\chart(\kk)=2$, we define
 $$
  \lnormed{\aa_1}{\aa_n}\circ \aa=
\begin{cases}
 \Lnormeda{\aa_1,\aa_2}{\aa,\aa_3\wdots \aa_{n-1}}{\aa_n}, &   \mbox{if } \aa\geq \aa_2, \\
 \lnormed{\aa_1,\aa,\aa_2}{\aa_{n-1}}\cdot \aa_n- \Lnormeda{\aa_2,\aa}{\aa_1,\aa_3\wdots\aa_{n-1}}{\aa_n}, &   \mbox{if } \aa< \aa_2.
\end{cases}
$$

\ITEM9 If~$\WW_1=\lnormed{\aa_1}{\aa_{n-1}}\cdot \aa_n$ with~$\nn\geq 4$, then we define
$$
(\lnormed{\aa_1}{\aa_{n-1}}\cdot \aa_n)\circ \aa=0.
$$
Moreover, if~$\chart(\kk)=2$, then we define
$$
\{\lnormed{\aa_1}{\aa_{n-1}}\cdot \aa_n,\aa\}=
\begin{cases}
 \Lnormeda{\aa_1,\aa_2}{\aa_3\wdots \aa_n}{\aa}, &   \mbox{if } \aa\geq \aa_2, \\
 \lnormed{\aa_1,\aa,\aa_2}{\aa_{n-1}}\cdot \aa_n- \Lnormeda{\aa_2,\aa}{\aa_3\wdots\aa_n}{\aa_1}, &   \mbox{if } \aa< \aa_2.
\end{cases}
$$
if~$\chart(\kk)\neq 2$, then we have~$\nn=4$, and we define
$$
\{\lnormed{\aa_1}{\aa_{3}}\cdot \aa_4,\aa\}=0.
$$

\ITEM{10} Finally, for every~$\WW$ in~$\mpb$ such that~$\ell(\WW)\geq 2$, for every~$\aa$ in~$\XX$, define
$$
\aa\circ \WW=\WW\circ \aa
$$ and define
$$
\{\WW,\aa\}=-\{\aa,\WW\}.
$$
\end{defi}

We shall soon show that~$(\kk\mpb, \circ, \{-,-\})$ is a metabelian Poisson algebra.
We first show that~$(\kk\mpb, \circ)$ is an associative commutative algebra.

\begin{lemm}\label{become-comm}
  With respect to Definition~\ref{defi-prod}, the space~$(\kk\mpb, \circ)$ becomes an associative commutative algebra.
\end{lemm}
\begin{proof}
It is enough to show that the commutativity and associativity hold on~$\mpb$.
We first prove the commutativity. For all~$\WW_1$ and~$\WW_2$ in~$\mpb$, if~$\WW_1$ or~$\WW_2$ lies in~$\XX$, then by Definition~\ref{defi-prod}\ITEM2 and~\ITEM{10}, we have~$\WW_1\circ \WW_2=\WW_2\circ \WW_1$; if~$\ell(\WW_1)\geq 2$ and~$\ell(\WW_2)\geq 2$, then by Definition~\ref{defi-prod}\ITEM1  we obtain~$\WW_1\circ \WW_2=\WW_2\circ \WW_1=0$.

Now we prove the associativity. For all~$\aa,\bb$ and~$\cc$ in~$\XX$, then
by Definition~\ref{defi-prod}\ITEM2 and~\ITEM3,  we have
$$
(\aa\circ \bb)\circ \cc=\ova{\aa\cdot\bb\cdot \cc}=\aa\circ (\bb\circ \cc).
$$
 For all~$\WW_1,\WW_2$ and~$\WW_3$ in~$\mpb$ such that~$\ell(\WW_1)+\ell(\WW_2)+\ell(\WW_3)\geq 4$, we obtain
 $$
 (\WW_1\circ \WW_2)\circ \WW_3=0=\WW_1\circ (\WW_2\circ \WW_3).
 $$
Therefore, $(\kk\mpb, \circ)$ is an associative commutative algebra.
\end{proof}

By Definition~\ref{defi-prod}, we can  show that~$(\kk\mpb, \{-,-\})$ is a Lie algebra. Some obvious formulas are as below.
\begin{lemm}\label{lie-anti-com}
  With respect to Definition~\ref{defi-prod},   we obtain

  \ITEM1 For all~$\WW_1$ and~$\WW_2$ in~$\mpb$, the equation~$\{\WW_1,\WW_2\}=-\{\WW_2,\WW_1\}$ holds.

  \ITEM2 For all~$\aa,\bb,\cc$ in~$\XX$, the equation~$\{\{\aa,\bb\},\cc\}+\{\{\bb,\cc\},\aa\}+\{\{\cc,\aa\},\bb\}=0$ holds.
\end{lemm}

\begin{proof}
  \ITEM1 For all~$\aa,\bb$ in~$\XX$, by Definition~\ref{defi-prod}\ITEM2, we obtain~$\{\aa,\bb\}=-\{\bb,\aa\}$ and~$\{\aa,\aa\}=0$. If~$\ell(\WW_1)\geq 2$ and~$\WW_2$ lies in~$\XX$, then we obtain~$\{\WW_1,\WW_2\}=-\{\WW_2,\WW_1\}$ by Definition~\ref{defi-prod}\ITEM{10}. Finally, if~$\ell(\WW_1)\geq 2$ and~$\ell(\WW_2)\geq 2$, then we have~$\{\WW_1,\WW_2\}=0=-\{\WW_2,\WW_1\}$ by Definition~\ref{defi-prod}\ITEM1.

  \ITEM2 By \ITEM1, we may assume that~$\aa\leq\bb\leq \cc$.  If~$\aa<\bb<\cc$, then we obtain
$$
   \{\{\aa,\bb\},\cc\}+\{\{\bb,\cc\},\aa\}+\{\{\cc,\aa\},\bb\}
   =-\Lnormed{\bb,\aa,\cc}-\Lnormed{\cc,\aa,\bb}+\Lnormed{\bb,\aa,\cc}+\Lnormed{\cc,\aa,\bb} =0.
$$
If~$\aa=\bb$ or~$\bb=\cc$, then by~\ITEM1 again, it is clear that~$\{\{\aa,\bb\},\cc\}+\{\{\bb,\cc\},\aa\}+\{\{\cc,\aa\},\bb\}=0$.
\end{proof}

Our second step towards Lemma~\ref{become-poi} is the following lemma, which is essential in showing that the Leibniz identity and Jacobi identity hold.
\begin{lemm}\label{many-comm}
With respect to Definition~\ref{defi-prod}, for every~$\WW$ in~$\mpb$ with~$\ell(\WW)\geq 2$,  for all~$\bb,\cc$ in~$\XX$, we have

  \ITEM1 $\{\WW,\bb\}\circ \cc=-\{\WW, \cc\}\circ \bb$.

  \ITEM2 $\{\WW\circ \bb,\cc\}=\{\WW,\cc\}\circ \bb$.

  \ITEM3 $\{\{\WW,\bb\},\cc\}=\{\{\WW, \cc\},\bb\}$.
\end{lemm}

\begin{proof} In the proof we shall use~$\aa_1\wdots\aa_n,...$ for letters in~$\XX$. The proofs are straightforward but there are many cases.

\ITEM1  If~$\WW=\aa_1\cdot \aa_2$ or~$\WW=\aa_1\cdot \aa_2\cdot \aa_3$ or~$\WW=\Lnormed{\aa_1\wdots\aa_n}\cdot \aa_{n+1}$ with~$\nn\geq 2$, then it is clear that~$\{\WW,\bb\}\circ \cc=0=-\{\WW,\cc\}\circ \bb$. Therefore, we may assume that~$\WW=\Lnormed{\aa_1\wdots\aa_n}$ and~$\bb\leq  \cc$.

 Assume first that~$\chart(\kk)\neq 2$. If~$\WW=\Lnormed{\aa_1\wdots\aa_n}$ lies in~$\mpb$ satisfying~$\nn\geq 3$, then we obtain~$\{\WW,\bb\}\circ\cc=0=-\{\WW,\cc\}\circ \bb$ by Definition~\ref{defi-prod}\ITEM6 and~\ITEM8. Assume~$\WW=[\aa_1,\aa_2]_{_\mathsf{L}}$ and~$\aa_1>\aa_2$. Then for this special case, we shall resort to Definition~\ref{defi-prod}\ITEM4 and~\ITEM6. (Since there are too many cases in Definition~\ref{defi-prod}\ITEM6, we add the needed equations above the equality symbol =, hoping to make it easy for the readers.)

We first consider the case~$\bb=\cc$. For this case, it is enough to show~$\{[\aa_1,\aa_2]_{_\mathsf{L}},\bb\}\circ \bb=0$. If~$\aa_2\leq \bb$, then we obtain
$$
\{[\aa_1,\aa_2]_{_\mathsf{L}},\bb\}\circ \bb=\Lnormed{\aa_1,\aa_2,\bb}\circ \bb
\overset{\eqref{a3=a}}{=}0 ;
$$
 if~$\aa_2>\bb$, then
 we obtain
 $$
 \{[\aa_1,\aa_2]_{_\mathsf{L}},\bb\}\circ \bb
 =\Lnormed{\aa_1,\bb,\aa_2}\circ \bb-\Lnormed{\aa_2,\bb,\aa_1}\circ \bb
\overset{\eqref{a1>a3>a},\eqref{a3max}}{=}-\Lnormed{\aa_2,\bb,\bb}\cdot \aa_1+\Lnormed{\aa_2,\bb,\bb}\cdot \aa_1
 =0.
 $$
 We now consider the case when~$\bb<\cc$.
 If~$\cc> \aa_1$, or, if~$\cc= \aa_1$ and~$\bb\geq\aa_2$,  then we obtain
$$\{[\aa_1,\aa_2]_{_\mathsf{L}},\bb\}\circ\cc
\overset{\eqref{amax}}{=}\{\{\aa_1,\aa_2\},\bb\}\cdot\cc
\overset{\eqref{a3max}, \eqref{a3=a1>aa2}}{=}-\Lnormed{\aa_1,\aa_2,\cc}\circ \bb
=-\{\{\aa_1,\aa_2\},\cc\}\circ \bb;
$$
if~$\cc= \aa_1$ and~$\bb<\aa_2$, then we have
\begin{multline*}
\{[\aa_1,\aa_2]_{_\mathsf{L}},\bb\}\circ \cc
=\Lnormed{\aa_1,\bb,\aa_2}\circ\cc -\Lnormed{\aa_2,\bb,\aa_1}\circ\cc\\
\overset{\eqref{amax}\eqref{a3=a}}{=}\Lnormed{\aa_1,\bb,\aa_2}\cdot\cc
\overset{\eqref{a3=a1>a2>a}}{=}-\Lnormed{\aa_1,\aa_2,\cc}\circ \bb
=-\{\Lnormed{\aa_1,\aa_2},\cc\}\circ \bb;
\end{multline*}
if~$\cc<\aa_1$ and~$\aa_2\leq \bb$, then we obtain
\begin{multline*}
 \{[\aa_1,\aa_2]_{_\mathsf{L}},\bb\}\circ\cc
 =\Lnormed{\aa_1,\aa_2,\bb}\circ\cc
 \overset{\eqref{a1>a>a3}}{=}\{\{\cc,\bb\},\aa_2\}\cdot\aa_1\\
 =-\{\{\bb,\cc\},\aa_2\}\cdot\aa_1
\overset{\eqref{a1>a3>a}}{=}-\Lnormed{\aa_1,\aa_2,\cc}\circ\bb
 =-\{[\aa_1,\aa_2]_{_\mathsf{L}},\cc\}\circ\bb;
\end{multline*}
 if~$\cc<\aa_1$ and~$\cc\geq\aa_2>\bb$, then we obtain
\begin{align*}
\{[\aa_1,\aa_2]_{_\mathsf{L}},\bb\}\circ\cc
=&\Lnormed{\aa_1,\bb,\aa_2}\circ\cc-\Lnormed{\aa_2,\bb,\aa_1}\circ\cc&\\
\overset{\eqref{a3=a},\eqref{a1>a>a3},\eqref{a3max}}{=}&\{\{\cc,\aa_2\},\bb\}\cdot\aa_1+\{\{\aa_2,\bb\},\cc\}\cdot\aa_1& \\
=&-\{\{\bb,\cc\},\aa_2\}\cdot\aa_1 \ \mbox{(by Lemma~\ref{lie-anti-com}\ITEM2)}&\\
\overset{\eqref{a1>a3>a}}{=}&-\Lnormed{\aa_1,\aa_2,\cc}\circ\bb
=-\{[\aa_1,\aa_2]_{_\mathsf{L}},\cc\}\circ\bb.&
\end{align*}
Finally, if~$\aa_2>\cc$ (in particular, we have~$\bb<\cc<\aa_2<\aa_1$), then we obtain
\begin{multline*}
\{[\aa_1,\aa_2]_{_\mathsf{L}},\bb\}\circ\cc
=\Lnormed{\aa_1,\bb,\aa_2}\circ\cc-\Lnormed{\aa_2,\bb,\aa_1}\circ\cc
\overset{\eqref{a1>a3>a},\eqref{a3max}}{=}\{\{\cc,\aa_2\},\bb\}\cdot\aa_1+\{\{\aa_2,\bb\},\cc\}\cdot\aa_1\\
 =-\{\{\aa_2,\cc\},\bb\}\cdot\aa_1-\{\{\bb,\aa_2\},\cc\}\cdot\aa_1
\overset{\eqref{a3max},\eqref{a1>a3>a}}{=}
\Lnormed{\aa_2,\cc,\aa_1}\circ\bb-\Lnormed{\aa_1,\cc,\aa_2}\circ\bb
=-\{[\aa_1,\aa_2]_{_\mathsf{L}},\cc\}\circ\bb.
\end{multline*}

Now we assume~$\chart(\kk)=2$.  Then we may assume~$\bb<\cc$. If~$\aa_2\leq \bb$, then we obtain
$$
  \{\Lnormed{\aa_1\wdots\aa_n},\bb\}\circ \cc+\{\Lnormed{\aa_1\wdots\aa_n},\cc\}\circ \bb
=\Lnormeda{\aa_1,\aa_2}{\aa_3\wdots\aa_n,\bb}{\cc}
-\Lnormeda{\aa_1,\aa_2}{\aa_3\wdots\aa_n,\cc}{\bb}=0;
$$
if~$\bb<\aa_2\leq \cc$, then we obtain
\begin{align*}
&\{\Lnormed{\aa_1\wdots\aa_n},\bb\}\circ \cc+\{\Lnormed{\aa_1\wdots\aa_n},\cc\}\circ \bb&\\
=&\Lnormed{\aa_1,\bb,\aa_2\wdots \aa_n}\circ\cc -\Lnormed{\aa_2,\bb,\ova{\aa_3\wdots\aa_n,\aa_1}} \circ\cc \\
&+\Lnormed{\aa_1, \aa_2, \ova{\aa_3\wdots\aa_n,\cc}}\circ \bb\\
=&\Lnormeda{\aa_1,\bb}{\aa_2\wdots \aa_n}{\cc}-\Lnormeda{\aa_2,\bb}{\aa_3\wdots\aa_n,\aa_1}{\cc}\\
&+\Lnormeda{\aa_1,\bb}{\aa_2\wdots \aa_n}{\cc}-\Lnormeda{\aa_2,\bb}{\aa_3\wdots\aa_n,\aa_1}{\cc}=0;&
\end{align*}
if~$\bb<\cc<\aa_2$, then we obtain
\begin{align*}
& \{\Lnormed{\aa_1\wdots\aa_n},\bb\}\circ \cc+\{\Lnormed{\aa_1\wdots\aa_n},\cc\}\circ \bb&\\
=&\Lnormeda{\aa_1,\bb}{\aa_2\wdots \aa_n}{\cc}
-\Lnormeda{\aa_2,\bb}{\aa_3\wdots\aa_n,\aa_1}{\cc}&\\
&+\Lnormed{\aa_1,\cc,\ova{\aa_2\wdots \aa_n}} \circ \bb
-\Lnormed{\aa_2,\cc,\ova{\aa_3\wdots\aa_n,\aa_1}} \circ \bb&\\
=&\Lnormeda{\aa_1,\bb}{\aa_2\wdots \aa_n}{\cc}-\Lnormeda{\aa_2,\bb}{\aa_3\wdots\aa_n,\aa_1}{\cc}
+\Lnormeda{\aa_1,\bb}{{\aa_2\wdots \aa_n}}{\cc}&\\
&-\Lnormeda{\cc,\bb}{{\aa_2\wdots \aa_n}}{\aa_1}
-\Lnormeda{\aa_2,\bb}{\aa_3\wdots\aa_n,\aa_1}{\cc}
+\Lnormeda{\cc,\bb}{{\aa_2\wdots \aa_n}}{\aa_1}=0.&
\end{align*}
\ITEM2  If~$\WW=\aa_1\cdot \aa_2$ or~$\WW=\aa_1\cdot \aa_2\cdot \aa_3$
or~$\WW=\Lnormed{\aa_1\wdots\aa_{n-1}}\cdot \aa_n$ with~$\nn\geq 2$, then we obtain~$\{\WW\circ \bb,\cc\}=0=\{\WW,\cc\}\circ \bb$.
If~$\WW=[\aa_1,\aa_2]_{_\mathsf{L}}$ with~$\aa_1>\aa_2$, then the result follows immediately from Definition~\ref{defi-prod}.
We now assume that~${\WW_1=\lnormed{\aa_1}{\aa_n}}$ with~$\nn\geq 3$.
If~$\chart(\kk)\neq 2$, then we obtain~$\{\WW_1\circ \bb,\cc\}=\{\WW_1,\cc\}\circ \bb=0$;
If~$\chart(\kk)=2$, then the strategy is to discuss according to which one is a minimal element among~$\aa_2$, $\bb$ and~$\cc$.

For the case when~$\aa_2$ is minimal among~$\aa_2$, $\bb$ and~$\cc$, that is, $\aa_2\leq \cc$ and~$\aa_2\leq \bb$~hold, we obtain
\begin{align*}
 &\{\lnormed{\aa_1}{\aa_n}\circ \bb,\cc\}
 =\{\Lnormeda{\aa_1,\aa_2}{\aa_3\wdots\aa_{n}}{\bb},\cc\}&\\
 =&\Lnormeda{\aa_1,\aa_2}{\aa_3\wdots\aa_{n},\bb}{\cc}
 =\Lnormeda{\aa_1,\aa_2}{\aa_3\wdots\aa_{n},\cc}{\bb}
 =\{\lnormed{\aa_1}{\aa_n},\cc\}\circ \bb.&
\end{align*}

For the case when~$\bb$ is minimal among~$\aa_2$, $\bb$ and~$\cc$. There are several subcases: If~$\bb<\aa_2\leq \cc$ holds, then we obtain
\begin{align*}
& \{\lnormed{\aa_1}{\aa_n}\circ \bb,\cc\}&\\
 =&\{\lnormed{\aa_1,\bb,\aa_2}{\aa_{n-1}}\cdot \aa_{n},\cc\}-\{\Lnormeda{\aa_2,\bb}{\aa_3\wdots\aa_n}{\aa_1},\cc\}&\\
  =&\{\Lnormeda{\aa_1,\bb}{\aa_2\wdots\aa_{n}}{\cc} -\Lnormeda{\aa_2,\bb}{\cc,\aa_3\wdots\aa_{n}}{\aa_1}&\\
  =&\Lnormed{\aa_1,\aa_2,\ova{\aa_3\wdots\aa_{n},\cc}}\circ \bb
  =\{\lnormed{\aa_1}{\aa_n},\cc\}\circ \bb;&
\end{align*}
if~$\bb<\cc<\aa_2$ holds, then we obtain
\begin{align*}
 &\{\lnormed{\aa_1}{\aa_n}\circ \bb,\cc\}
 =\{\lnormed{\aa_1,\bb,\aa_2}{\aa_{n-1}}\cdot \aa_{n},\cc\}-\{\Lnormeda{\aa_2,\bb}{\aa_3\wdots\aa_n}{\aa_1},\cc\}&\\
   =&\Lnormed{\aa_1,\bb,\cc,\aa_2\wdots\aa_{n-1}}\cdot \aa_n -\Lnormeda{\aa_2,\bb,\cc}{\aa_3\wdots\aa_{n}}{\aa_1}&\\
  = &\Lnormed{\aa_1,\bb,\cc,\aa_2\wdots\aa_{n-1}}\cdot \aa_n -\Lnormeda{\cc,\bb}{\aa_2\wdots\aa_{n}}{\aa_1}&\\
  &-\Lnormeda{\aa_2,\bb,\cc}{\aa_3\wdots\aa_{n}}{\aa_1}
  +\Lnormeda{\cc,\bb}{\aa_2\wdots\aa_{n}}{\aa_1}&\\
  = &\Lnormed{\aa_1,\cc,\aa_2\wdots\aa_n}\circ \bb
  - \Lnormed{\aa_2,\cc,\ova{\aa_3\wdots\aa_{n},\aa_1}}\circ\bb
  =\{\lnormed{\aa_1}{\aa_n},\cc\}\circ \bb;&
\end{align*}
If~$\bb=\cc<\aa_2$, then similar to the case when~$\bb<\cc<\aa_2$,  we can obtain the desired formula.

Finally, the proof for the case when~$\cc< \bb$ and~$\cc<\aa_2$ simultaneously hold   is similar to that for the case when~$\bb$ is minimal among~$\aa_2$, $\bb$ and~$\cc$.

\ITEM3 Without loss of generality, we assume that~$\bb<\cc$. If~$\WW=\aa_1\cdot \aa_2\cdot\aa_3$, then we obtain~$\{\{\WW,\bb\},\cc\}=0=\{\{\WW,\cc\},\bb\}$.
For~$\WW=\aa_1\cdot \aa_2$, then by~\ITEM1\ITEM2 and by Lemma~\ref{lie-anti-com}, we obtain
\begin{align*}
& \{\{\aa_1\cdot \aa_2,\bb\},\cc\}-\{\{\aa_1\cdot \aa_2,\cc\},\bb\}&\\
=&\{\{\aa_1,\bb\}\circ \aa_2,\cc\}+  \{\{\aa_2,\bb\}\circ\aa_1,\cc\}
 -\{\{\aa_1,\cc\}\circ\aa_2,\bb\}- \{\{\aa_2,\cc\}\circ \aa_1,\bb\}&\\
 =&\{\{\aa_1,\bb\},\cc\}\circ \aa_2+  \{\{\aa_2,\bb\},\cc\}\circ \aa_1
 -\{\{\aa_1,\cc\},\bb\}\circ \aa_2- \{\{\aa_2,\cc\},\bb\}\circ \aa_1 \ \mbox{(by \ITEM2)}&\\
 =&-\{\{\bb,\cc\},\aa_1\} \circ \aa_2 -\{\{\bb,\cc\},\aa_2\} \circ \aa_1&\\
 =&\{(\cc,\bb),\aa_1\} \circ \aa_2 +\{(\cc,\bb),\aa_2\} \circ \aa_1=0 \ \mbox{(by \ITEM1)}.&
\end{align*}
For~$\WW=\Lnormed{\aa_1\wdots\aa_n}$ with~$\nn\geq 2$, the proof is similar to that for~\ITEM2. More precisely,
 if~$\cc<\aa_2$~holds, then we obtain
\begin{align*}
 &\{\{\lnormed{\aa_1}{\aa_n},\bb\},\cc\}
 =\{\lnormed{\aa_1,\bb,\aa_2}{\aa_{n}},\cc\}
 -\{\Lnormed{\aa_2,\bb,\ova{\aa_3\wdots\aa_n,\aa_1}},\cc\}&\\
   =&\Lnormed{\aa_1,\bb,\cc,\aa_2\wdots\aa_{n}}
   -\Lnormed{\aa_2,\bb,\cc,\ova{\aa_3\wdots\aa_n,\aa_1}}&\\
  = &\Lnormed{\aa_1,\bb,\cc,\aa_2\wdots\aa_{n}} -\Lnormed{\cc,\bb,\ova{\aa_2\wdots\aa_{n},\aa_1}}
  -\Lnormed{\aa_2,\bb,\cc,\ova{\aa_3\wdots\aa_n,\aa_1}}
  +\Lnormed{\cc,\bb,\ova{\aa_2\wdots\aa_{n},\aa_1}}&\\
  =& \{\Lnormed{\aa_1,\cc,\aa_2\wdots\aa_n}, \bb\}
  - \{\Lnormed{\aa_2,\cc,\ova{\aa_3\wdots\aa_{n},\aa_1}},\bb\}
  =\{\{\lnormed{\aa_1}{\aa_n},\cc\}, \bb\};&
\end{align*}
If~$\bb<\aa_2\leq \cc$~holds, then we obtain
\begin{align*}
 &\{\{\lnormed{\aa_1}{\aa_n},\bb\},\cc\}
 =\{\lnormed{\aa_1,\bb,\aa_2}{\aa_{n}},\cc\}
 -\{\Lnormed{\aa_2,\bb,\ova{\aa_3\wdots\aa_n,\aa_1}},\cc\}&\\
   =&\Lnormed{\aa_1,\bb,\ova{\cc,\aa_2\wdots\aa_{n}}}
   -\Lnormed{\aa_2,\bb,\ova{\cc,\aa_3\wdots\aa_n,\aa_1}} &\\
  =&\{\Lnormed{\aa_1,\aa_2,\ova{\cc,\aa_3\wdots\aa_n}} , \bb\}
  =\{\{\Lnormed{\aa_1\wdots\aa_n},\cc\}, \bb\};&
\end{align*}
If~$\aa_2\leq \bb$~holds, then we obtain
$$ \{\{\lnormed{\aa_1}{\aa_n},\bb\},\cc\}
  =\Lnormed{\aa_1,\aa_2,\ova{\bb,\cc,\aa_3\wdots\aa_n}}
  =\{\{\Lnormed{\aa_1\wdots\aa_n},\cc\}, \bb\}.$$
Finally, assume~$\WW=\Lnormed{\aa_1\wdots\aa_{n}}\cdot \aa_{n+1}$ and~$\nn\geq 2$. If~$\chart(\kk)\neq 2$, then it is obvious that~$\{\{\WW,\bb\},\cc\}=0=\{\{\WW,\cc\},\bb\}$. If~$\chart(\kk)=2$, then
we have~$$\WW=\Lnormed{\aa_1\wdots\aa_{n}}\cdot \aa_{n+1}
=\Lnormed{\aa_1\wdots\aa_{n}}\circ\aa_{n+1}.$$ So by~\ITEM2 and by the above reasoning, we obtain
$$\{\{\WW,\bb\},\cc\}=\{\{\Lnormed{\aa_1\wdots\aa_{n}},\bb\},\cc\}\circ\aa_{n+1}
=\{\{\Lnormed{\aa_1\wdots\aa_{n}},\cc\},\bb\}\circ\aa_{n+1}
=\{\{\WW,\cc\},\bb\}.$$
The proof is completed.
\end{proof}

Now we are ready to prove that the Jacobi identity holds in~$(\kk\mpb, \{-,-\})$. Moreover, we have the following lemma.

\begin{lemm}\label{become-poi}
With respect to Definition~\ref{defi-prod}, the space~$(\kk\mpb,\circ, \{-,-\})$ becomes a metabelian Poisson algebra generated by a well-ordered set~$\XX$.
\end{lemm}
\begin{proof}
We first show that~$(\kk\mpb,\circ, \{-,-\})$ is a Lie algebra.
By Lemma~\ref{lie-anti-com} and by Definition~\ref{defi-prod}, for all~$\WW_1$ and~$\WW_2$ in~$\mpb$, we obtain~$\{\WW_1,\WW_2\}=-\{\WW_2,\WW_1\}$ and~$\{\WW_1,\WW_1\}=0$. Therefore, for every~$\xx=\sum_{1\leq \ii\leq n}\alpha_{i}\WW_i$ in~$\kk\mpb$, we have
$$\{\xx,\xx\}=\sum_{1\leq i,j\leq n}\alpha_{i}\alpha_{j}\{\WW_i,\WW_j\}=0.$$

As for the Jacobi identity, similar to the above reasoning, it is enough to show that, for all~$\WW_1,\WW_2$ and~$\WW_3$ in~$\mpb$, we have
\begin{equation}\label{Jacobi}
\{\{\WW_1,\WW_2\},\WW_3\}-\{\{\WW_1,\WW_3\},\WW_2\}=-\{\{\WW_2,\WW_3\},\WW_1\}.
\end{equation}
Without loss of generality, we may assume that~$\ell(\WW_1)\geq \ell(\WW_2)\geq \ell(\WW_3)$.
By Definition~\ref{defi-prod}, if~$\ell(\WW_1)\geq \ell(\WW_2)\geq 2$,  then it is clear that~\eqref{Jacobi} holds.
Assume now that~$\WW_2=\bb$ and~$\WW_3=\cc$ lie in~$\XX$. Then by Lemma~\ref{lie-anti-com}\ITEM2 and by Lemma~\ref{many-comm}\ITEM3, Eq.~\eqref{Jacobi} holds.

By Lemma~\ref{become-comm}, to prove that~$(\kk\mpb,\circ, \{-,-\})$ is a Poisson algebra, it remains to show that~$(\kk\mpb,\circ, \{-,-\})$ satisfies the Leibniz identity. Moreover, it is enough to show that, for all~$\WW_1,\WW_2,\WW_3$ in~$\mpb$, we have
\begin{equation}\label{leibnizp}
  \{\WW_1\circ \WW_2,\WW_3\}=\{\WW_1,\WW_3\}\circ \WW_2+\{\WW_2,\WW_3\}\circ \WW_1.
\end{equation}
By Lemma~\ref{become-comm}, we may assume~$\ell(\WW_1)\geq\ell(\WW_2)$.
 If two of~$\ell(\WW_1),\ell(\WW_2)$ and~$\ell(\WW_3)$ are at least 2, or, if~$\ell(\WW_1)=\ell(\WW_2)=\ell(\WW_3)=1$, then by Definition~\eqref{defi-prod}, we obtain Eq.~\eqref{leibnizp} immediately.
Assume~$\ell(\WW_1)\geq2$ and~$\ell(\WW_2)=\ell(\WW_3)=1$. Then by Definition~\ref{defi-prod}, we have~$\{\WW_2,\WW_3\}\circ \WW_1=0$, and we obtain~$\{\WW_1\circ \WW_2,\WW_3\}=\{\WW_1,\WW_3\}\circ \WW_2$ by Lemma~\ref{many-comm}\ITEM2.  Finally, assume~$\ell(\WW_3)\geq 2$ and~$\ell(\WW_1)=\ell(\WW_2)=1$, then by Lemma~\ref{lie-anti-com}\ITEM1 and by Lemma~\ref{many-comm}\ITEM1, we obtain~$\{\WW_1,\WW_3\}\circ \WW_2+\{\WW_2,\WW_3\}\circ \WW_1=0$, and by Definition~\ref{defi-prod}, we obtain~$\{\WW_1\circ \WW_2,\WW_3\}=0$. Therefore,  Eq.~\eqref{leibnizp} holds.

Finally, for all~$\WW_1$ and~$\WW_2$ in~$\mpb$ satisfying~$\ell(\WW_1)\geq2$ and~$\ell(\WW_2)\geq2$, we have~$\WW_1\circ\WW_2=\{\WW_1,\WW_2\}=0$. Therefore, by Lemma~\ref{meta-equiv-def}, $(\kk\mpb, \{-,-\})$ is a metabelian Poisson algebra.
\end{proof}

We are now in a position to show that~$\mpb$ forms a linear basis for a free metabelian Poisson algebra generated by the well-ordered set~$\XX$.

\begin{theorem}\label{iso}
With respect to Definition~\ref{defi-prod}, the algebra~$(\kk\mpb,\circ,\{-,-\})$ is isomorphic to the free metabelian Poisson algebra~$\MP\XX$ generated by~$\XX$. In particular, the set~$\mpb$ forms a linear basis for~$\MP\XX$.
\end{theorem}
\begin{proof}
By Lemma~\ref{generating-set}, $\mpb$ is a linear generating set for~$\MP\XX$.
Now define a homomorphism~$\varphi: \MP\XX \longrightarrow \kk\mpb$, induced by~$\varphi(\aa)=\aa$
for every~$\aa$ in~$\XX$. Then for every~$\WW$ in~$\mpb$, we have~$\varphi(\WW)=\WW$ and~$\varphi(\mpb)=\mpb$. In particular, the homomorphism~$\varphi$ is an isomorphism, and thus the set~$\mpb$ is a linear basis for~$\MP\XX$.
\end{proof}

\begin{rema}\label{iden-rema}
Because of Theorem~\ref{iso}, we can identify~$\MP\XX$ with~$\kk\mpb$. In particular, Definition~\ref{defi-prod} offers the multiplication table of elements in~$\mpb$. Moreover, to show that~$\mpb$ is a linear basis of the free metabelian Poisson algebra generated by~$\XX$, it is enough to assume that~$\XX$ is a linear ordered set. But in the next section, when we are considering the \gsb, we have to assume that~$\XX$ is a well-ordered set.
\end{rema}

To conclude this subsection, we observe that, the Freiheitssatz for metabelian Poisson algebras does not hold in general. For instance, let~$\ff=(\bb,\aa)\cdot\bb-\aa$, where~$\aa,\bb$ are letters and~$\aa\prec\bb$. Then we have~$\ff\cdot \aa=\aa\cdot \aa$. So the subalgebra of~$\PMP{\{\aa,\bb\}}{\{\ff\}}$ generated by~$\{\aa\}$ is not a free metabelian Poisson algebra.

\section{Finitely presented metabelian Poisson algebras}\label{main-result}
Our aim in this section is to show that, the word problem for every finitely presented metabelian Poisson algebra is solvable. We shall elaborate a Gr\"{o}bner--Shirshov basis method for metabelian Poisson algebras,
which begins with defining normal~$\SS$-polynomials satisfying good properties. We also remark that Gr\"{o}bner--Shirshov bases methods for algebras are in general not unique. And, the principle of our way of elaborating a Gr\"{o}bner--Shirshov basis method for metabelian Poisson algebras is to avoid discussions as many as possible.
For~$\SS \subseteq \MP\XX$, we denote by~$\PMP\XX\SS$ the metabelian Poisson algebra presented by a set~$\XX$ of generators and a set~$\SS$ of defining relations, that is, by definition, the quotient~$\MP\XX{/}\Id\SS$, where $\Id\SS$ is the ideal of~$\MP\XX$ generated by~$\SS$.

\subsection{Normal~$\SS$-polynomials}
We first introduce a well order on~$\mpb$. Though the set~$\mpb$ depends on~$\chart(\kk)$, we want to offer the same rule for comparing elements in~$\mpb$. For every~$\WW$ in~$\mpb$, denote by~$\pb(\WW)$ the number of Poisson brackets that appear in~$\WW$, and denote by~$\ell(\WW)$ the \emph{length} of~$\WW$, that is, the number of letters (with repetitions) that appear in~$\WW$.
For monomial~$\WW=\Lnormed{\aa_1\wdots\aa_n}$ or~$\WW=\Lnormed{\aa_1\wdots\aa_{n-1}}\cdot \aa_{n}$ or~$\WW=\aa_1...\aa_n$ in~$\mpb$, define
$$\wt(\WW)=(\ell(\WW),\pb(\WW),\aa_1\wdots\aa_n).$$
Finally, for all~$\WW_1$ and~$\WW_2$ in~$\mpb$, define~$\WW_1<\WW_2$ if~$\wt(\WW_1)<\wt(\WW_2)$ lexicographically. It is clear that~$<$ is a well order on~$\mpb$.
 \begin{exam}
 For all~$\aa,\bb,\cc,\aa_1\wdots \aa_{n+1}$ in~$\XX$ satisfying~$\aa_2\leq \pdots\leq \aa_{n+1}$, $\nn\geq 1$, we have the followings:

  \ITEM1 $\ell(\aa_1\cdot \aa_1)=2$  and~$\pb(\aa_1\cdot \aa_1)=0$.

  \ITEM2 $\ell(\Lnormed{\aa_1\wdots\aa_n})=n$ and~$\pb(\Lnormed{\aa_1\wdots\aa_n})=n-1$.

   \ITEM3 $\ell(\Lnormed{\aa_1\wdots\aa_n}\cdot\aa_{n+1})=n+1$ and~$\pb(\Lnormed{\aa_1\wdots\aa_n}\cdot\aa_{n+1})=n-1$ if~$\Lnormed{\aa_1\wdots\aa_n}\cdot\aa_{n+1}$ lies in~$\mpb$.

   \ITEM4 If~$\aa<\bb<\cc$ holds, then we have~$\aa\cdot \bb\cdot \cc<(\bb,\aa)\cdot \cc<(\cc,\aa)\cdot \bb<(\cc,\bb)\cdot \aa<\Lnormed{\bb,\aa,\cc}<\Lnormed{\cc,\aa,\bb}$.
 \end{exam}

For every~$\ff=\sum_i\alpha_i\WW_i$ in~$\MP\XX$, where each~$\alpha_i\neq 0$ lies in~$\kk$, each~$\WW_i$ lies in~$\mpb$ and satisfies~$\WW_1>\WW_2>\cdots$, we call~$\WW_1$ the \emph{leading monomial} of~$\ff$, and denote it by~$\bar\ff$; we call~$\alpha_1$ the \emph{leading coefficient} of~$\ff$, and denote it by~$\lcoe\ff$. A polynomial~$\ff$ is called \emph{monic} if~$\lcoe\ff=1$, and, a subset~$\SS$ of~$\MP\XX$ is called monic if every element of~$\SS$ is monic. Finally, define~$\bar0=0$ and define~$0<\WW$ for every~$\WW$ in~$\mpb$.

We observe that the order~$<$ is a monomial order in the following sense.
\begin{lemm}\label{monomial-order}
Suppose~$\WW\in\mpb$ and assume~$\wt(\WW)=(\ell(\WW),\pb(\WW),\bb_1\wdots\bb_m)$. If~$\ell(\WW)\geq5$ holds, then for all~$\aa_1\wdots\aa_n$ in~$\XX$,
 we have
$$\wt(\ov{\Lnormed{\WW,\aa_1\wdots\aa_n}})=(\ell(\WW)+\nn,\pb(\WW)+\nn,\bb_1, \ova{\bb_2\wdots\bb_m,\aa_1\wdots\aa_n}).$$
In particular, for all~$\WW_1$ and~$\WW_2$ in~$\mpb$ satisfying~$\ell(\WW_1)\geq 5$ and~$\WW_2<\WW_1$, for all~$\aa_1\wdots\aa_n$ in~$\XX$, we have~$\ov{\Lnormed{\WW_2,\aa_1\wdots\aa_n}}<\ov{\Lnormed{\WW_1,\aa_1\wdots\aa_n}}$.
\end{lemm}

\begin{proof}
We first assume~$\nn=1$.
If~$\WW=\Lnormed{\bb_1\wdots\bb_m}$, then by Remark~\ref{iden-rema} and by Definition~\ref{defi-prod}, we obtain
$$(\WW,\aa_1)=
\begin{cases}
\Lnormed{\bb_1,\bb_2,\ova{\bb_3\wdots\bb_{m},\aa_1}}, &   \mbox{if } \aa_1\geq \bb_2, \\
\lnormed{\bb_1,\aa_1,\bb_2}{\bb_m}- \Lnormed{\bb_2,\aa_1,\ova{\bb_3\wdots\bb_m,\bb_1}}, &   \mbox{if } \aa_1< \bb_2.
\end{cases}
$$
Since~$\bb_1>\bb_2$, we have~$\ov{(\WW,\aa_1)}=\Lnormed{\bb_1,\ova{\bb_2\wdots\bb_{m},\aa_1}}$. In particular, the polynomial~$(\WW,\aa_1)$ is monic and we obtain
$$\wt(\ov{(\WW,\aa_1)})=(\ell(\WW)+1,\pb(\WW)+1,\bb_1,\ova{\bb_2\wdots\bb_m,\aa_1}).$$
If~$\WW=\Lnormed{\bb_1\wdots\bb_{m-1}}\cdot \bb_m$, then since~$\mm\geq 5$ and~$\WW$ in~$\mpb$, we deduce~$\chart(\kk)=2$. Therefore, we obtain
$$ (\WW, \aa_1)=
\begin{cases}
 \Lnormeda{\bb_1,\bb_2}{\bb_3\wdots \bb_m}{\aa_1}, &   \mbox{if } \aa_1\geq \bb_2, \\
 \lnormed{\bb_1,\aa_1,\bb_2}{\bb_{n-1}}\cdot \bb_m- \Lnormeda{\bb_2,\aa_1}{\bb_3\wdots\bb_m}{\bb_1}, &   \mbox{if } \aa_1< \bb_2.
\end{cases}
$$
Again, since~$\bb_1>\bb_2$, we have~$\ov{(\WW,\aa_1)}=\Lnormeda{\bb_1}{\bb_2\wdots\bb_{m}}{\aa_1}$.
In particular, we obtain
$$\wt(\ov{(\WW,\aa_1)})
=(\ell(\WW)+1,\pb(\WW)+1,\bb_1,\ova{\bb_2\wdots\bb_m,\aa_1}).$$

For the second claim, assume~$\WW_1,\WW_2\in\mpb$ and assume~$\ell(\WW_1)\geq 5$ and~$\WW_2<\WW_1$.
If we have~$\ell(\WW_2)<\ell(\WW_1)$, then we deduce
$$\ell(\ov{(\WW_2,\aa_1)})<\ell(\WW_1)+1=\ell(\ov{(\WW_1,\aa_1)})\
(\mbox{if}~(\WW_2,\aa_1)\neq 0),$$
and thus we have
$$\ov{(\WW_2,\aa_1)}<\ov{(\WW_1,\aa_1)}.$$
If~$\ell(\WW_2)=\ell(\WW_1)$ holds, then we assume~$\wt(\WW_2)=(\mm,\pb(\WW_2),\bb_1'\wdots\bb_m')$.
By the above reasoning, we have
$$\wt(\ov{(\WW_2,\aa_1)})=(\mm+1,\pb(\WW_2)+1,\bb_1',\ova{\bb_2'\wdots\bb_m',\aa_1}).$$
Since~$\wt(\WW_2)<\wt(\WW_1)$, we immediately obtain~$\wt(\ov{(\WW_2,\aa_1)})<\wt(\ov{(\WW_1,\aa_1)})$.
That is, the lemma holds for the case~$\nn=1$. For~$\nn>1$, since~$\ell(\ov{(\WW_1,\aa)})\geq 5$ and~$\wt(\ov{(\WW_2,\aa_1)})<\wt(\ov{(\WW_1,\aa_1)})$, the lemma follows by induction on~$\nn$.
\end{proof}

We are now ready to define normal~$\SS$-polynomials, which are special polynomials in~$\Id\SS$ playing an important role concerning the word problem for~$\PMP\XX\SS$.
\begin{defi}\label{defi-nsp}
For every~$\SS\subseteq\MP\XX$, we define \nsps\ as follows:

\ITEM1 Every~$\ss$ in~$\SS$ is a \nsp.

\ITEM2 For every~$\ss$ in~$\SS$ such that~$\ell(\bar\ss)\geq 5$, for all~$\aa_1\wdots\aa_n$ in~$\XX$,
the polynomial~$\Lnormed{\ss,\aa_1\wdots\aa_n}$ is also a \nsp.
\end{defi}

Recall that Definition~\ref{defi-prod} offers the multiplication table for~$\mpb$. In particular, for the product~$\Lnormed{\aa_1,\aa_2,\aa_3}\circ\aa$ (which is
 also~$\Lnormed{\aa_1,\aa_2,\aa_3}\cdot\aa$ in~$\MP\XX$)  in Definition~\ref{defi-prod}\ITEM6, there are many cases to consider, and thus, it is very lengthy to discuss the leading monomial of~$\ss\cdot\aa$ in~$\MP\XX$, especially when~$\ell(\ov\ss)\leq 4$. However,  it is very important in general to consider the leading monomial of a polynomial when one deals with the Gr\"{o}bner--Shirshov bases method for metabelian Poisson algebras. This again forces us to avoid discussions and to define normal~$\SS$-polynomials in the way of Definition~\ref{defi-nsp}.

Note  that the leading monomial of a \nsp\ satisfies the following good property.

\begin{lemm}\label{leading-monomial}
Let~$\ss$ be a monic polynomial in~$\SS$ and assume~$\wt(\bar\ss)=(\ell(\bar\ss),\pb(\bar\ss),\bb_1\wdots\bb_m)$. If~$\mm\geq 5$ holds, then for all~$\aa_1\wdots\aa_n$ in~$\XX$,  we
 have  $$\ov{\Lnormed{\ss,\aa_1\wdots\aa_n}}=\ov{\Lnormed{\bar\ss,\aa_1\wdots\aa_n}}$$
and
$$\wt(\ov{\Lnormed{\ss,\aa_1\wdots\aa_n}})=(\ell(\bar\ss)+\nn,\pb(\bar\ss)+\nn,\bb_1, \ova{\bb_2\wdots\bb_m,\aa_1\wdots\aa_n}),$$
in particular, the \nsp~$\Lnormed{\ss,\aa_1\wdots\aa_n}$ is also monic.
\end{lemm}
\begin{proof}
By Lemma~\ref{monomial-order}, we immediately obtain the result.
\end{proof}
By Lemma~\ref{leading-monomial}, for every~$\ss$ in~$\SS$,
 if~$\ell(\bar\ss)\geq 5$ holds, then we have~$\ov{\Lnormed{\ss,\aa_1\wdots\aa_n}}\geq \ov\ss$.
We are now in a position to introduce the notion of a \gsb\ for a metabelian Poisson algebra. In fact, the way of defining a \gsb\ for an algebra is quite simple and is just a routine job. What is really different is to find a sufficient (and necessary) condition for a set of being a \gsb.
\begin{defi}\label{defi-gsb}
Let~$I$ be an ideal of~$\MP\XX$. Then a set~$\SS\subseteq\MP\XX$ is called a \emph{Gr\"{o}bner--Shirshov basis} (in~$\MP\XX$) for the quotient algebra~$\MP\XX/I$, if~$I=\Id\SS$, and for every nonzero element~$\ff\in I$, there exists a normal~$\SS$-polynomial~$h$ satisfying~${\ffb=\ov{h}}$.
\end{defi}

For every subset~$\SS$ of~$\MP\XX$, we define
$$
\mathsf{Irr}(S):=\{\WW\in \mpb \mid \WW\neq \ov{\hh} \mbox{ for every normal } S\mbox{-polynomial } \hh\}.
$$
We observe that the set~$\mathsf{Irr}(S)$ is a linear generating set of the metabelian Poisson algebra~$\PMP\XX\SS$. More precisely, we have the following lemma:
\begin{lemm}\label{f=irr+n-s-polynomials}
Let~$\SS$ be a monic subset of~$\MP\XX$. Then every polynomial~$\ff$ in~$\MP\XX$ can be written of the form:
$$\ff=\sum\alpha_i\WW_i+
\sum\beta_jh_{j},
$$
where all~$\alpha_i$ and~$\beta_j$ are elements in~$ k$, each~$ \WW_i$ is a monomial in~$\mathsf{Irr}(\SS)$ satisfying~$\WW_i\leq \bar f$, and, each~$\hh_{j}$ is
a normal~$\SS$-polynomial satisfying~$\ov{h_{j}}\leq\bar \ff$. In particular, the set~$\Irr\SS$ is a linear generating set of the algebra~$\PMP\XX\SS$.
\end{lemm}

\begin{proof}
The result follows by induction on~$\bar f $: If~$\ffb=\bar\hh$ for some \nsp, then by Lemma~\ref{leading-monomial}, we have~$\ov{\ff-\lcoe\ff\hh}<\ffb$; if~$\ffb$ lies in~$\Irr\SS$, then~$\ov{\ff-\lcoe\ff\ffb}<\ffb$. Since~$<$ is a well order, the subtraction process terminates in finitely many iterations.
 \end{proof}

Note that it is not practical in general to tell whether a set~$\SS$ is a \gsb\ for~$\PMP\XX\SS$ by using Definition~\ref{defi-gsb}. So, what we need at this point is a family of tractable conditions for recognizing whether a given set is a \gsb\ or not. On one hand, the characteristic~$\chart(\kk)$ of the underlying field~$\kk$ matters, and on the other hand, the way of defining \nsps\ is not unique, so this sort of conditions is also not unique. Finally, since our aim is to show that every finitely presented metabelian Poisson algebra has a solvable word problem, we try to choose a way that avoids discussions as many as possible. We first introduce several specific polynomials in~$\Id\SS$, which is useful in rewriting elements of~$\Id\SS$ into \nsps\ under certain assumptions.

\begin{defi}
For every~$\ss$ in~$\SS\subseteq \MP\XX$, for all~$\aa_1,\aa_2$ and~$\aa$ in~$\XX$ satisfying~$\aa_1>\aa_2$, the polynomials~$\ss\cdot \aa$, $\ss\cdot (\aa_1,\aa_2)$ and~$(\ss,(\aa_1,\aa_2))$ are all called~\emph{multiplication compositions} (of~$\SS$), and, if~$\ell(\bar\ss)\leq 4$ holds, then~$(\ss,\aa)$ is also called a~\emph{multiplication composition} (of~$\SS$). A multiplication composition~$\hh$ is called \emph{trivial} (with respect to~$\SS$) if~$\hh$ can be written as a linear combination of \nsps\ with their leading monomials $\leq \bar\hh$. Finally, if every multiplication composition of~$\SS$ is trivial, then~$\SS$ is called \emph{trivial under multiplication compositions}.
\end{defi}

To make formulas simple, we introduce the following notations.
\begin{defi}
  A polynomial~$\ff$ in~$\Id\SS$ is called \emph{trivial modulo~$\SS$ with respect to~$\WW$}, denoted by
$$
f \equiv 0 \mbox{ modulo } \SS \mbox{ with respect to } \WW,
$$ if
$f=\sum_{i}\alpha_{i}\hh_{i}$ for some normal~$S$-polynomials~$\hh_{i}$'s satisfying~$\overline{\hh_{i}}< \WW$. We say~$\ff\equiv \gg$ modulo~$\SS$ with respect to~$\WW$ if  $\ff-\gg\equiv 0$ modulo~$\SS$ with respect to~$\WW$.
\end{defi}
 The first approach toward deciding whether a set~$\SS$ is a \gsb\ or not, is to decide whether all elements of~$\Id\SS$ can be written as linear combinations of \nsps. One of the basic observation is as follows.

 \begin{lemm}\label{mult-letter}
  Let~$\SS$ be a subset of~$\MP\XX$ that is trivial under multiplication compositions. Then  for all~$\aa_1\wdots\aa_n$ in~$\XX$,
 for every \nsp~$\ff$, the polynomial~$\Lnormed{\ff,\aa_1\wdots\aa_n}$ can be written as a linear combination of \nsps. In particular, for every~$\WW$ in~$\mpb$ satisfying~$\ell(\WW)\geq 5$, if~$\ffb<\WW$ holds, then we have~$\Lnormed{\ff,\aa_1\wdots\aa_n}\equiv 0$ modulo~$\SS$ with respect to~$\ov{\Lnormed{\WW,\aa_1\wdots\aa_n}}$.
 \end{lemm}
 \begin{proof}
If~$\ff=\Lnormed{\ss,\bb_1\wdots\bb_m}$ for some element~$\ss$ in~$\SS$
satisfying~$\ell(\bar\ss)\geq 5$ and~$\mm\geq 0$, then
$$
\Lnormed{\ff,\aa_1\wdots\aa_n}
=\Lnormed{\ss,\bb_1\wdots\bb_m,\aa_1\wdots\aa_n}
$$
 is a \nsp. By Lemmas~\ref{monomial-order} and~\ref{leading-monomial}, we
deduce~$\ov{\Lnormed{\ff,\aa_1\wdots\aa_n}}<\ov{\Lnormed{\WW,\aa_1\wdots\aa_n}}$.

Otherwise, assume~$\ff=\ss$ and assume~$\ell(\ov\ss)\leq 4$. We use induction on~$\nn$. For~$\nn=1$, since~$\SS$ is trivial under multiplication compositions, we can assume that~$(\ss,\aa_1)=\sum_{i}\alpha_i\ff_i$, where each~$\ff_i$ is a \nsp\ such
 that~$\ov{\ff_i}\leq\ov{(\ss,\aa_1)}<\ov{(\WW,\aa)}$. For~$\nn\geq 1$, by the above reasoning and by induction hypothesis, we obtain the result.
 \end{proof}
Before going further, we note that in a metabelian Poisson algebra, we have the following formulas.
\begin{lemm}\label{moving-letter}
 Let~$\mathcal{P}$ be a metabelian Poisson algebra. Then for all~$\xx_1\wdots\xx_n$, $\xx$ and~$\yy$ in~$\mathcal{P}$, we have the following identities:

 \ITEM1 If~$\nn\geq 2$, then we have~$(\xx,\Lnormed{\xx_1\wdots\xx_n})=\Lnormed{\xx,(\xx_1,\xx_2),\xx_3\wdots\xx_n}$.

\ITEM2 For every~$\nn\geq 1$, we have~$\Lnormed{\xx,\xx_1\wdots\xx_{n}}\cdot \yy=\Lnormed{\xx\cdot \yy,\xx_1\wdots\xx_{n}}+\Lnormed{\xx\cdot(\xx_1,\yy),\xx_2\wdots\xx_{n}}.$
 \end{lemm}
 \begin{proof}
   \ITEM1 For~$\nn=2$, there is nothing to prove. Assume~$\nn>2$. Then by Eq.~\eqref{lie-commu},
   we deduce~
   $$
   (\xx,\Lnormed{\xx_1\wdots\xx_n})
   =-(\Lnormed{\xx_1\wdots\xx_n},\xx)
   =-\Lnormed{(\xx_1,\xx_2),\xx,\xx_3\wdots\xx_n}
   =\Lnormed{\xx,(\xx_1,\xx_2),\xx_3\wdots\xx_n}.
   $$

   \ITEM2 We use induction on~$\nn$. For~$\nn=1$, we have
   $$(\xx,\xx_1)\cdot \yy
   =(\xx\cdot \yy, \xx_1)-\xx\cdot (\yy,\xx_1)
   =(\xx\cdot \yy, \xx_1)+\xx\cdot (\xx_1,\yy).$$
   Assume~$\nn\geq 2$. Then by induction hypothesis, we obtain
   \begin{align*}
     &\Lnormed{\xx,\xx_1\wdots\xx_{n}}\cdot \yy
     =(\Lnormed{\xx,\xx_1\wdots\xx_{n-1}}\cdot \yy,\xx_n)
     +\Lnormed{\xx,\xx_1\wdots\xx_{n-1}}\cdot (\xx_n,\yy)&\\
     =&(\Lnormed{\xx,\xx_1\wdots\xx_{n-1}}\cdot \yy,\xx_n)
     = \Lnormed{\xx\cdot \yy,\xx_1\wdots\xx_{n}}+\Lnormed{\xx\cdot(\xx_1,\yy),\xx_2\wdots\xx_{n}}.&
   \end{align*}
   The proof is completed.
 \end{proof}


We are now ready to claim the condition under which the set of all the \nsps\ forms a linear generating set for~$\Id\SS$.

 \begin{lemm}\label{as-nsp}
Let~$\SS$ be a subset of~$\MP\XX$. If~$\SS$ is trivial under multiplication compositions, then every polynomial in~$\Id\SS$ can be written as a linear combination of \nsps.
 \end{lemm}
\begin{proof}
  Since every polynomial in~$\SS$ is a normal~$\SS$ polynomial, by the Leibniz identity, it is enough to show that, for every \nsp~$\ff$, for every~${\WW=\Lnormed{\aa_1\wdots\aa_n}}$ in~$\mpb$ satisfying~$\nn\geq 1$, both~$(\ff,\WW)$ and~$\ff\cdot\WW$ can be written as linear combinations of \nsps.

  We first assume~$\ff=\ss$. If~$\nn=1$ and~$\ell(\bar\ss)\geq 5$, then~$(\ss,\aa_1)$ is already a \nsp, and~$\ss\cdot\aa$ is a multiplication composition which by assumption can be written as a linear combination of \nsps. If~$\nn=1$ and~$\ell(\bar\ss)\leq 4$, or, if~$\nn=2$, then both~$(\ss,\WW)$ and~$\ss\cdot\WW$ are just multiplication compositions, and thus, they can be written as linear combinations of \nsps.
  Finally, assume~$\nn\geq3$, then by Lemma~\ref{moving-letter}, we obtain
  $$
  (\ss,\Lnormed{\aa_1\wdots\aa_n})=\Lnormed{\ss,(\aa_1,\aa_2), \aa_3\wdots\aa_n}
  $$
  and
$$
\ss\cdot\Lnormed{\aa_1\wdots\aa_n}
=\Lnormed{\aa_1\cdot \ss,\aa_2\wdots\aa_n}+\Lnormed{\aa_1\cdot(\aa_2,\ss),\aa_3\wdots\aa_n}
=\Lnormed{\ss\cdot (\aa_1,\aa_2),\aa_3\wdots\aa_n}.
$$
Since both~$(\ss,(\aa_1,\aa_2))$ and~$\ss\cdot(\aa_1,\aa_2)$ are multiplication compositions, by assumption and by Lemma~\ref{mult-letter}, we know that both~$(\ss,\WW)$ and~$\ss\cdot\WW$ can be written as linear combinations of \nsps.

Now we assume~$\ff=\Lnormed{\ss,\bb_1\wdots\bb_m}$ for some letters~$\bb_1\wdots\bb_m$ in~$\XX$ and assume~$\mm\geq 1$. Then we have~$\ff\cdot\WW=(\ff,\WW)=0$ if~$\ell(\WW)\geq 2$. So we may assume~$\WW=\aa\in\XX$, and thus~$
(\Lnormed{\ss,\bb_1\wdots\bb_m},\WW)$ is already a normal~$\SS$-polynomial.
Moreover, we have
$$
\Lnormed{\ss,\bb_1\wdots\bb_m}\cdot\WW=
\Lnormed{\ss\cdot \WW,\bb_1\wdots\bb_m}+\Lnormed{\ss\cdot(\bb_1,\WW),\bb_2\wdots\bb_m}.
$$
By Lemma~\ref{mult-letter} again, the lemma follows.
\end{proof}

\subsection{Composition-Diamond lemma for metabelian Poisson algebras}
Our aim in this subsection is to prove a Composition-Diamond lemma for metabelian Poisson algebras, which offers several equivalent conditions of~$\SS$ being a \gsb\ for~$\PMP\XX\SS$. In the previous subsection, we find a condition under which the set of~\nsps\ forms a linear generating set of~$\Id\SS$. Our next step is to find a condition under which the difference of two \nsps\ with the same leading monomial is not essential in the sense of Lemma~\ref{key-lemma}.

Before going there, we observe that, if~$\SS$ is trivial under multiplication compositions, then the difference of the \nsps~$\Lnormed{\ss,\aa_1\wdots\aa_n}$ and~$\Lnormed{\ss,\ova{\aa_1\wdots\aa_n}}$ is not essential in the following sense.

\begin{lemm}
Let~$\SS$ be a subset of~$\MP\XX$ that is trivial under multiplication compositions. For every~$\ss$ in~$\SS$, if~$\ell(\bar\ss)\geq 5$ holds,  then for all~$\aa_1\wdots\aa_n$ in~$\XX$,
we have
$$\Lnormed{\ss,\aa_1\wdots\aa_n}
-\Lnormed{\ss,\ova{\aa_1\wdots\aa_n}}\equiv 0$$ modulo~$\SS$ with respect to~$\ov{\Lnormed{\ss,\aa_1\wdots\aa_n}}$.
\end{lemm}

\begin{proof}
By Eq.~\eqref{lie-commu}, we have~$\Lnormed{\ss,\aa_1\wdots\aa_n}
=\Lnormed{\ss,\aa_1,\ova{\aa_2,\wdots\aa_n}}$. So we assume~$\aa_2\leq \pdots\leq\aa_n$. If~$\aa_1\leq \aa_2$, then there is nothing to prove. If~$\aa_1>\aa_2$, then we obtain
$$
\Lnormed{\ss,\aa_1\wdots\aa_n}-\Lnormed{\ss,\aa_2,\aa_1,\aa_3\wdots\aa_n}
=\Lnormed{\ss,(\aa_1,\aa_2),\aa_3\wdots\aa_n}.
$$
Since~$\SS$ is trivial under multiplication compositions, we obtain~$(\ss,(\aa_1,\aa_2))=\sum_{i}\alpha_i\ff_i$, where each~$\ff_i$ is a \nsp\ satisfying~$\bar{\ff_i}\leq \ov{(\ss,(\aa_1,\aa_2))}$. Since we have~$(\WW_1,\WW_2)=0$ for all~$\WW_1$ and~$\WW_2$ in~$\mpb$ satisfying~$\ell(\WW_1)\geq2$ and~$\ell(\WW_2)\geq 2$, if~$(\ss,(\aa_1,\aa_2))\neq 0$, then we deduce~$\ell(\ov{(\ss,(\aa_1,\aa_2))})=3$. In particular, we obtain~$\bar{\ff_i}< \ov{\Lnormed{\ss,\aa_1,\aa_2}}$. Finally, by Lemma~\ref{mult-letter}, the lemma follows.
\end{proof}

 Now we are ready to introduce general compositions (Definition~\ref{defi-composition}), which are special polynomials in the ideal~$\Id\SS$ playing an important role in deciding whether the set~$\SS$ of relations is a \gsb\ for~$\PMP\XX\SS$ or not.

 We first recall the notion of the least common multiple of two elements in the free commutative monoid~$[\XX]$ (freely) generated by a set~$\XX$ (the multiple is not in~$\MP\XX$ in general). For all~$\WW_1,\WW_2$ and~$\WW_3$ in the free commutative monoid~$[\XX]$, if~$\WW_1\cdot\WW_2=\WW_3$, then~$\WW_3$ is called a multiple of~$\WW_1$ and~$\WW_2$ in~$[\XX]$. For~$\aa_1...\aa_n,\bb_1...\bb_m\in[\XX]$, where each~$\aa_i,\bb_j$ lies in~$\XX$, if
$$\aa_1...\aa_n\aa_1'...\aa_p'=\bb_1...\bb_m\bb_1'...\bb_q'$$ in~$[\XX]$ for some letters~$\aa_1'\wdots\aa_p',\bb_1'\wdots\bb_q'$  in~$\XX$ ($\pp,\qq\geq 0$), and none of the letters~$\aa_1'\wdots\aa_p'$ appears in the letters~$\bb_1'\wdots\bb_q'$, then
the monomial~$\aa_1...\aa_n\aa_1'...\aa_p'$ is called the least common multiple of~$\aa_1...\aa_n$ and~$\bb_1...\bb_m$.

\begin{defi}\label{defi-composition}
 Let~$\SS$ be a monic subset of~$ \MP\XX$. For all~$\ss_1$ and~$\ss_2$ in~$\SS$, we define \emph{general compositions} as follows:

 \ITEM1 If we have~$\ov{\ss_1}=\ov{\ss_2}$, then the polynomial~$\ss_1-\ss_2$ is a \emph{general composition} of~$\SS$ with respect to~$\ov{\ss_1}$.

 \ITEM2 If~$\ov{\ss_1}=\Lnormed{\aa,\aa_1\wdots\aa_n}$
 and~$\ov{\ss_2}=\Lnormed{\aa,\bb_1\wdots\bb_m}$ for some letters~$\aa_1\wdots\aa_n,\bb_1\wdots\bb_m$ in~$\XX$ satisfying~$\nn\geq4$ and~$\mm\geq4$, assuming that~$\aa_1...\aa_n\aa_1'...\aa_p'=\bb_1...\bb_m\bb_1'...\bb_q'$ is the least common multiple of~$\aa_1...\aa_n$ and~$\bb_1...\bb_m$ in the free commutative monoid generated by~$\XX$, where~$\aa_1'\wdots\aa_p',\bb_1'\wdots\bb_q'$ are letters in~$\XX$, then
 $$\Lnormed{\ss_1,\ova{\aa_1'\wdots\aa_p'}}
 -\Lnormed{\ss_2,\ova{\bb_1'\wdots\bb_q'}}$$ is a \emph{general composition} of~$\SS$ with respect to~$\ov{\Lnormed{\ss_1,\aa_1'\wdots\aa_p'}}$.

 \ITEM3 If~$\ov{\ss_1}=\Lnormed{\aa,\aa_1\wdots\aa_{n-1}}\cdot\aa_n$
 and~$\ov{\ss_2}=\Lnormed{\aa,\bb_1\wdots\bb_{m-1}}\cdot \bb_m$ for~$\aa_1\wdots\aa_n,\bb_1\wdots\bb_m$ in~$\XX$ satisfying~$\nn\geq4$ and~$\mm\geq4$, assuming that~$\aa_1...\aa_n\aa_1'...\aa_p'
 =\bb_1...\bb_m\bb_1'...\bb_q'$ is the least common
 multiple of~$\aa_1...\aa_n$ and~$\bb_1...\bb_m$ in the free commutative monoid generated by~$\XX$, where~$\aa_1'\wdots\aa_p',\bb_1'\wdots\bb_q'$ are letters in~$\XX$, then $$\Lnormed{\ss_1,\ova{\aa_1'\wdots\aa_p'}}
 -\Lnormed{\ss_2,\ova{\bb_1'\wdots\bb_q'}}$$  is a \emph{general composition} of~$\SS$ with respect to~$\ov{\Lnormed{\ss_1,\aa_1'\wdots\aa_p'}}$.

 If for every~$\WW$ in~$\mpb$, for every composition~$\ff$ of~$\SS$ with respect to~$\WW$ (if any), we have~$\ff=\sum_i\alpha_i\ff_i$ for some \nsps~$\ff_\ii$'s satisfying~$\bar{\ff_i}<\WW$, then~$\SS$ is called \emph{trivial under general compositions}.
\end{defi}
What we simply call ``compositions'' in the sequel are multiplication compositions and general compositions. Moreover, for a general composition~$\hh$ with respect to~$\WW$, we said that the composition~$\hh$ is trivial (with respect to~$\SS$) if~$\hh\equiv 0$ modulo~$\SS$ with respect to~$\WW$.
\begin{lemm}\label{key-lemma}
Let~$\SS$ be a monic subset of~$\MP\XX$ such that~$\SS$ is trivial under compositions. Then for all \nsps~$\hh_1$ and~$\hh_2$ satisfying~$\ov{\hh_1}=\ov{\hh_2}$, we have~$\hh_1-\hh_2\equiv 0$ modulo~$\SS$ with respect to~$\ov{\hh_1}$.
\end{lemm}

\begin{proof}
 If~$\hh_1=\ss_1$ and~$\hh_2=\ss_2$, then~$\hh_1-\hh_2$ is exactly a general composition and thus by assumption we have~$\hh_1-\hh_2\equiv 0$ modulo~$\SS$ with respect to~$\ov{\hh_1}$.

If~$\hh_1=\Lnormed{\ss_1,\aa_1\wdots\aa_n}$ and~$\hh_2=\Lnormed{\ss_2,\bb_1\wdots\bb_m}$ for some letters~$\aa_1\wdots\aa_n,\bb_1\wdots\bb_m$ with~$\nn\geq 0$ and~$\mm\geq 0$, then suppose
$$
\aa_1...\aa_n
=\aa_1'...\aa_l'\cc_1...\cc_p \mbox{ and }
\bb_1...\bb_m
=\aa_1'...\aa_l'\dd_1...\dd_q
$$
in~$[\XX]$, where none of the letters~$\cc_1\wdots\cc_p$ appears in the letters~$\dd_1\wdots\dd_q$. Then the polynomial~$\Lnormed{\ss_1,\ova{\cc_1\wdots\cc_p}}-\Lnormed{\ss_2,\ova{\dd_1\wdots\dd_q}}$ is a general composition and therefore, we may assume that~$\Lnormed{\ss_1,\ova{\cc_1\wdots\cc_p}}-\Lnormed{\ss_2,\ova{\dd_1\wdots\dd_q}}=\sum_i\alpha_i\gg_i$ for some \nsps~$\gg_i$'s such that~$\ov{\gg_i}<\ov{\Lnormed{\ss_1,\cc_1\wdots\cc_p}}$. Moreover, by Lemma~\ref{mult-letter}, we obtain
\begin{multline*}
  \hh_1-\hh_2=\Lnormed{\ss_1,\aa_1\wdots\aa_n}-\Lnormed{\ss_2,\bb_1\wdots\bb_m}
  \equiv \Lnormed{\ss_1,\ova{\aa_1\wdots\aa_n}}-\Lnormed{\ss_2,\ova{\bb_1\wdots\bb_m}}\\
  \equiv \Lnormed{\ss_1,\ova{\cc_1\wdots\cc_p},\aa_1'\wdots\aa_l'}
  -\Lnormed{\ss_2,\ova{\dd_1\wdots\dd_q},\aa_1'\wdots\aa_l'}
  \equiv\sum\nolimits_i\alpha_i\Lnormed{\gg_i,\aa_1'\wdots\aa_l'}
  \equiv 0
\end{multline*}
modulo~$\SS$ with respect to~$\ov{\hh_1}$.
\end{proof}
With the previous results, the proof of the CD-lemma is now a routine job, almost the same as in the case of associative algebras. For the convenience of the readers, we quickly repeat the argument.

\begin{theorem}\label{cd-lemma}{\bf (Composition-Diamond lemma for metabelian Poisson algebras) }Let~$\SS $ be a monic subset of~$\MP\XX$. Then the followings are equivalent:

\ITEM1 The set~$\SS$ is trivial under compositions.

\ITEM2 The set~$\SS$ is a \gsb\ for~$\PMP\XX\SS$.

\ITEM3 The set~$\mathsf{Irr}(S):=\{\WW\in \mpb \mid \WW\neq \ov{\hh} \mbox{ for every normal } S\mbox{-polynomial } \hh\}$ is a linear basis of~$\PMP\XX\SS$.
\end{theorem}

\begin{proof}
 \ITEM1 $\Rightarrow$ \ITEM2  Let~$\Id\SS$ be the ideal of~$\MP\XX$ generated by~$\SS$. For every nonzero element~$ \ff$ in~$\Id\SS$, by Lemma~\ref{as-nsp},
 we may assume~$ \ff=\sum_{i=1}^n\alpha_i\hh_{i}$, where each~$\hh_{i}$ is a normal~$\SS$-polynomial and each~$\alpha_i$ is a nonzero element in~$\kk$.
Let~$\WW_i= \overline{\hh_{i}}$. Then we assume~$\WW_1=\WW_2=\pdots=\WW_l>\WW_{l+1}\geq \pdots$. We use induction on~$\WW_1$ to show that~$\ov\ff$ is the same as a leading monomial of a normal~$\SS$-pollynomial. For~$\WW_1=\ov{\ff}$, there is nothing to prove. For~$\WW_1>\ov{\ff}$, we have~$\sum_{i=1}^l\alpha_{i}=0$ and
 \begin{multline*}
 \ff=\sum_{i=1}^l\alpha_i\hh_{i}+\sum_{ i=l+1}^n\alpha_i\hh_{i}
 =\sum_{ i=1}^l\alpha_i\hh_{1}
 -\sum_{ i=1}^l\alpha_i(\hh_{1}
 -  \hh_{i})+\sum_{ i=l+1}^n\alpha_i\hh_{i}
 =0+\sum_{j }\beta_{j}\hh_{j}'+\sum_{ i=l+1}^n\alpha_i\hh_{i},
 \end{multline*}
 where each normal~$\SS$-polynomial~$\hh_j'$ satisfies that~$\ov{\hh_{j}'}< \WW_1$ by Lemma~\ref{key-lemma}.
Point~\ITEM2 follows by induction hypothesis.

\ITEM2 $\Rightarrow$ \ITEM3 By Lemma~\ref{f=irr+n-s-polynomials},
$\Irr\SS$ is a linear generating set for~$\PMP\XX\SS$. Suppose that $\sum_{i}\alpha_i\WW_i=0$ in~$\PMP\XX\SS$, where each element~$ \alpha_i\in \kk$ is nonzero and each monomial~$\WW_i\in \Irr\SS $ satisfies that~$\WW_1>\WW_2> \pdots$. Then, as an element in~$\MP\XX$, the polynomial~$\sum_{i}\alpha_i\WW_i$ is a nonzero element in~$\Id\SS $.  But then we have~$\overline{\sum_{i}\alpha_i\WW_i}=\WW_1$ in~$ \Irr\SS $, which contradicts with Point~\ITEM2.

\ITEM3 $ \Rightarrow$ \ITEM1 All multiplication compositions and general compositions of~$\SS$ are elements of~$\Id\SS$. By Lemma~\ref{f=irr+n-s-polynomials} and~\ITEM3, we obtain~\ITEM1.
 \end{proof}
\begin{rema}\label{shirshov-algo}
   We also recall the method how one can obtain a \gsb\ by calculating compositions in general. For every nonempty monic subset~$\SS$ of~$\MP\XX$, let~$\SS_1=\SS$. We calculate the compositions of~$\SS_1$, if all the compositions are trivial, then~$\SS_1$ is already a \gsb\ and we set~$\SS_2=\SS_1$;  Otherwise, we add all the nontrivial compositions into~$\SS_1$ to obtain a new set~$\SS_2$. (If a nontrivial composition~$\ff$ is not monic, then we just add~$\lcoe{\ff}^{-1}\ff$.) Assume that we already obtain~$\SS_1\wdots\SS_n$, if~$\SS_n$ is already a~\gsb, then we set~$\SS_{n+1}=\SS_n$; otherwise, we calculate all the compositions of~$\SS_n$ and add all the non-trivial ones into~$\SS_n$ to obtain a new set~$\SS_{n+1}$. Continue in this way, and we set~$\SS'=\cup_{i\geq 1}\SS_i$. Then we claim that~$\SS'$ is a \gsb\ for~$\PMP\XX\SS$. In fact, every composition of~$\SS'$ is a composition of~$\SS_n$ for some integer~$\nn$ and thus is trivial by construction.
\end{rema}
Compared with the paper on Gr\"{o}bner--Shirshov bases method for metabelian Lie algebras~\cite{met-lie-cd-lemma}, we can find that our new method avoids lots of discussions at the price of probably calculating more multiplication compositions when the length of the leading monomial of the involved polynomial in~$\SS$ is less than 5.
Moreover, the case for metabelian Poisson algebras is more complicated than that for metabelian Lie algebras, and so, we believe that it is better to avoid lengthy discussions, especially when our aim is to consider the word problem for finitely presented metabelian Poisson algebras but not to calculate the \gsb~for some concrete examples.

\subsection{Finitely presented metabelian Poisson algebra}
We are now ready to deal with our main topic: word problem for finitely presented metabelian Poisson algebras. Since~$\mpb$ is a linear basis of the free metabelian Poisson algebra generated by~$\XX$, and elements of~$\mpb$ are of the form~$\Lnormed{\aa_1,\ova{\aa_2\wdots\aa_n}}$ or of the form~$\Lnormeda{\aa_1}{\aa_2\wdots\aa_{n-1}}{\aa_n}$ except finitely many ones, it is natural to wonder whether the word problem for finitely presented metabelian Poisson algebra is solvable like finitely presented commutative algebras. We shall offer a positive answer in Theorem~\ref{finite-GSB}. Before going there, we first investigate more properties on \gsb, which are in fact general properties for Gr\"{o}bner--Shirshov bases for various kind of algebras.

\begin{lemm}\label{subset-gsb}
Let~$\SS$ be a monic \gsb\ for~$\PMP\XX\SS$, and let~$\RR$ be a subset of~$\SS$. If for every~$\ff$ in~$\SS$, there is a normal~$\RR$-polynomial~$\hh$  such that~$\bar\hh=\bar\ff$, then~$\RR$ is also a \gsb\ for~$\PMP\XX\SS$.
 \end{lemm}

 \begin{proof}
We first show that the leading monomial of an arbitrary nonzero polynomial in~$\Id\SS$ is a leading monomial of a normal~$\RR$-polynomial. For every nonzero element~$\ff$ in~$\Id\SS$, there is some \nsp~$\hh$ such that~$\bar\ff=\bar\hh$. If~$\hh=\ss$ for some element~$\ss$ in~$\SS$, then by assumption~$\bar\ff=\bar\ss$ is a leading monomial of some \nrp. Otherwise, we have~$\hh=\Lnormed{\ss,\aa_1\wdots\aa_n}$ for some element~$\ss$ in~$\SS$  and for some letters~$\aa_1\wdots\aa_n$ in~$\XX$ satisfying~$\ell(\bar\ss)\geq 5$ and~$\nn\geq 1$. Moreover, by assumption, we have~$\bar\ss=\bar\gg$ for some \nrp~$\gg=\Lnormed{\ss',\bb_1'\wdots\bb_m'}$ satisfying~$\ell(\ov{\ss'})\geq 5$. Finally, we obtain
$$\bar\ff=\ov{\Lnormed{\ss',\bb_1'\wdots\bb_m',\aa_1\wdots\aa_n}}.$$
Moreover, since~$\RR$ is a subset of~$\SS$, if~$\bar\ff=\bar\hh$ for some normal~$\RR$-polynomial~$\hh$, then~$\ff-\lcoe\ff\hh$ is still in~$\Id\SS$, and we have~$\ov{\ff-\lcoe\ff\hh}<\bar\ff$. It immediately follows that~$\ff$ can be written as a linear combination of normal~$\RR$-polynomials, in other words, we have~$\Id\RR=\Id\SS$.  Therefore,  
the set~$\RR$ is a \gsb\ for~$\PMP\XX\SS$.
 \end{proof}

To make description easier, we define the following notion on minimal \gsb. The results about minimal \gsb\ are just reminiscent of  those for commutative algebras. The proofs are quite easy, but for completeness, we still offer the proofs.
\begin{defi}
  A subset~$\SS$ of~$\MP\XX$ is called \emph{minimal}, if for every~$\ff$ in~$\SS$ and for every
  normal~$(\SS\setminus\{\ff\})$-polynomial~$\hh$, we have~$\bar\ff\neq \bar\hh$.
\end{defi}

Now we prove the existence of  minimal \gsb, that is, a set that is minimal and at the same time, it is a \gsb\ for a given algebra.

\begin{lemm}\label{minimal-gsb}
For every metabelian Poisson algebra~$\PMP\XX\SS$, there is a minimal \gsb~$\SS'$ for~$\PMP\XX\SS$.
\end{lemm}
\begin{proof}
We shall construct~$\SS'$ directly.  Without loss of generality, assume~$\SS$ is a \gsb. Then for every~$\WW=\bar\ss$ for some polynomial~$\ss$ in~$\SS$, we arbitrarily choose one polynomial~$\hh_\WW$ in~$\SS$ satisfying~$\ov{\hh_\WW}=\WW$. Let~$\RR$ be the set of the collection of all the polynomials~$\hh_\WW$'s, that is,
 $$\RR=\{\hh_\WW\mid \WW=\bar\ss~\mbox{for some polynomial}~\ss\in\SS, \WW\in \mpb\}.$$
Then different elements in~$\RR$ have different leading monomials. By Lemma~\ref{subset-gsb}, $\RR$ is a \gsb\ for~$\PMP\XX\SS$.

 For every~$\WW$ in~$\mpb$, we define a subset~$\RR_{\WW}$ of~$\RR$ inductively. Let~$\WW_0$ be the minimal element of the set of all the leading monomials of polynomials in~$\RR$ and define~$\RR_\WW=\emptyset$ for every~$\WW<\WW_0$ in~$\mpb$.  We also define
$$\RR_{\WW_0}=\{\ff\in\RR\mid \bar\ff=\WW_0\}.$$
Assume~$\RR_{\WW'}$ has already been defined for every~$\WW'$ in~$\mpb$ satisfying~$\WW'<\WW$. Then we define
$$
\RR_{<\WW}=\bigcup_{\WW'<\WW}\RR_{\WW'}
$$
and$$
\RR_{\WW}=\{\ff \in \RR\mid \ffb
=\WW \mbox{ and } \ffb\in \Irr{\RR_{<\WW}}\}.
$$
  Finally, we define
  $$\SS'=\bigcup_{\WW\in \mpb}\RR_{\WW}.$$
  We claim that~$\SS'$ is a minimal \gsb\ for~$\PMP\XX\SS$. By construction, it is clear that~$\SS'$ is a minimal set. Moreover, for every~$\ff$ in~$\RR$, say~$\ffb=\WW$, if~$\ff$ lies in~$\RR_{\WW}$, then~$\ff$ lies in~$\SS'$;
  Otherwise, we obtain~$\ffb=\bar\hh$ for some normal~$\RR_{<\WW}$-polynomial~$\hh$. By Lemma~\ref{subset-gsb}, $\SS'$ is a minimal \gsb\ for~$\PMP\XX\RR=\PMP\XX\SS$.
\end{proof}
 
We are now in a position to show that every finitely generated metabelian Poisson algebra can be finitely presented (that is, every finitely generated metabelian Poisson algebra is Noetherian), and the word problem for every finitely presented metabelian Poisson algebra is solvable.
\begin{theorem}\label{finite-GSB}
Let~$\XX$ be a finite set. Then~$\PMP\XX\SS$ has a finite \gsb. In particular, the word problem for an arbitrary finitely presented metabelian Poisson algebra is solvable.
 \end{theorem}
\begin{proof}
To prove the first claim, by Lemma~\ref{minimal-gsb}, we may assume that~$\SS$ is a minimal \gsb\ for~$\PMP\XX\SS$. Suppose that~$\SS$ is infinite
and define $$\SS_1=\{\ff\in \SS\mid \ell(\bar\ff)\geq 5\}.$$ Then since~$\SS$ is an infinite minimal set and~$\XX$ is finite, we deduce that~$\SS_1$ is also infinite. For every~$\aa$ in~$\XX$, define
$$\SS_{\aa}=\{\ff\in\SS\mid \ffb=\Lnormed{\aa,\aa_1\wdots\aa_n} \mbox{ for some letters } \aa,\aa_1\wdots\aa_n\in\XX,\nn\geq 4 \}$$ and for all~$\aa,\bb$ in~$\XX$, define$$\SS_{\aa,\bb}=\{\ff\in\SS\mid \ffb=\Lnormed{\aa,\aa_1\wdots\aa_n}\cdot \bb \mbox{ for some letters } \aa,\bb,\aa_1\wdots\aa_n\in\XX,\nn\geq 3 \}.$$
Then we obtain
$$
\SS_1=\bigcup_{\aa\in\XX}\SS_a\cup \bigcup_{\aa,\bb\in\XX}\SS_{\aa,\bb}.
$$
Since~$\SS_1$ is infinite and~$\XX$ is finite, some set~$\SS_a$ or~$\SS_{\aa,\bb}$ is infinite. Without loss of generality, assume that~$\SS_{\aa,\bb}$ is infinite. (Note that~$\SS_{\aa,\bb}$ is an empty set if~$\chart(\kk)\neq2$.) For all~$\ff_1$ and~$\ff_2$ in~$\SS_{\aa,\bb}$, suppose that
$$\bar{\ff_1}=\Lnormed{\aa,\aa_1\wdots\aa_n}\cdot \bb$$ and
$$\bar{\ff_2}=\Lnormed{\aa,\bb_1\wdots\bb_m}\cdot \bb.$$
Then for all~$\cc_1\wdots\cc_p$ in~$\XX$, for the \nsp~$\Lnormed{\ff_1,\cc_1\wdots\cc_m}$, by Lemma~\ref{leading-monomial}, we obtain
$$
\ov{\Lnormed{\ff_1,\cc_1\wdots\cc_m}}
=\Lnormeda{\aa}{\aa_1\wdots\aa_n,\cc_1\wdots\cc_p}{\bb}.
$$
Since~$\SS$ is minimal, we have
$$
\aa_1...\aa_n\cdot\cc_1...\cc_p
\neq \bb_1...\bb_m
$$
in the free commutative monoid~$[\XX]$. Define
$$
\SS_{\aa,\bb}'
=\{\aa_1...\aa_n \in[\XX]\mid \ff\in\SS, \ffb
=\Lnormed{\aa,\aa_1\wdots\aa_n}\cdot \bb,\nn
\geq 3,\aa,\bb,\aa_1\wdots\aa_n\in\XX\}.
$$ Then~$\SS_{\aa,\bb}'$ is an infinite set such that for all~$\WW_1$ and~$\WW_2$ in~$\SS_{\aa,\bb}'$, the monomial~$\WW_1$ is not a multiple of~$\WW_2$. But then the ideal generated by~$\SS_{\aa,\bb}'$ in the free commutative algebra~$\kk[\XX]$ is not finitely generated, which contradicts to the Hilbert's Basis Theorem. The reasoning for the case that~$\SS_a$ is infinite is similar to the above reasoning for the case that~$\SS_{\aa,\bb}$ is infinite. Therefore, the set~$\SS$ is a finite \gsb\ for~$\PMP\XX\SS$.

We now turn to the second claim and assume that~$\SS$ is a finite set. We claim that one can obtain a finite \gsb\ by computing compositions as in Remark~\ref{shirshov-algo} in finite steps. With the notations of Remark~\ref{shirshov-algo}, suppose that~$\SS'=\cup_{i}\SS_i$ is a \gsb\ for the finitely presented metabelian Poisson algebra~$\PMP\XX\SS$. By the first claim, there is a finite monic minimal \gsb\ $\RR=\{\ff_1\wdots\ff_m\}$ for~$\PMP\XX\SS$. Then by Theorem~\ref{cd-lemma}, for every~$\ff_i$ in~$\RR$, there is some normal~$\SS'$-polynomial~$\hh_i$ such that~$\ov{\ff_i}=\ov{\hh_i}$. If~$\hh_i=\Lnormed{\ss_\ii,\aa_{\ii,1}\wdots\aa_{\ii,n_i}}$ for some polynomial~$\ss_\ii$ in~$\SS'$ with~$\ell(\ov{\ss_\ii})\geq 5$ and for some letters~$\aa_{\ii,1}\wdots\aa_{\ii,\nn_i}$ in~$\XX$, or, if~$\hh_i=\ss_\ii$ for some polynomial~$\ss_\ii$ in~$\SS'$, then let~$\gg_i=\ss_\ii$. Suppose that~$\gg_1\wdots\gg_m$ lie in~$\SS_{t}$, then it is clear that~$\SS_t$ is a \gsb. Moreover, for every element~$\ff$ in~$\MP\XX$, if~$\ffb=\bar\hh$ for some normal~$\SS_t$-polynomial~$\hh$, then~$\ff$ can be reduced to~$\ff-\lcoe\ff\hh$, and we have~$\ov{\ff-\lcoe\ff\hh}<\bar\ff$. Since~$<$ is a well order, the subtraction process  must terminate after finitely many iterations. Finally, the polynomial~$\ff=0$ in~$\PMP\XX\SS$ holds if and only if~$\ff$ is reduced to~$0$ by the above subtraction process.
 \end{proof}

 \subsection{Conditions for an endomorphism of~$\MP\XX$ to be an isomorphism} In this subsection, we will apply the Gr\"obner--Shirshov bases theory for metabelian Poisson algebras to study the automorphisms of~$\MP\XX$, where~$\XX$ is a finite well-ordered set.

 First of all, we observe that for a finitely generated free metabelian Poisson algebra~$\MP\XX$, an endomorphism~$\varphi$ of~$\MP\XX$ is an isomorphism if and only if~$\varphi$ is an epimorphism. This is a general fact, but for the convenience of the readers, we quickly offers a proof.

 \begin{lemm}\label{lemm-epi}
Let~$\varphi$ be an endomorphism of a free metabelian Poisson algebra~$\MP\XX$ generated by a finite set~$\XX$. Then~$\varphi$ is an isomorphism if~$\varphi$ is an epimorphism.
\end{lemm}
\begin{proof}
Suppose that~$\varphi$ is not an injection and let~$\ff$ be a nonzero polynomial satisfying~$\varphi(\ff)=0$ such that~$\ell(\ov\ff)$ is minimal among such polynomials. Assume~$\nn=\ell(\ov\ff)$ and let~$I$ be the ideal of~$\MP\XX$ generated by all the monomials in~$\mpb$ of length~$\geq\nn+1$. Then~$\varphi$ induces an epimorphism~$\widetilde{\varphi}$ from the quotient algebra~$\MP\XX/I$ to~$\MP\XX/I$, which is automatically an isomorphism because~$\MP\XX/I$ is finite dimensional. However, $\ff+I$ is a nonzero element in~$\MP\XX/I$ satisfying~$\widetilde{\varphi}(\ff+I)=I$, this contradicts with the fact that~$\widetilde{\varphi}$ is an isomorphism.
\end{proof}
 
Now we introduce a notation to denote the linear part of a polynomial. For every polynomial
$$\ff=\sum_{1\leq \tt\leq \pp}\alpha_\tt\WW_\tt+\sum_{\pp+1\leq \tt\leq \qq}\alpha_\tt\WW_\tt\in\MP\XX$$
  satisfying
  $$\ell(\WW_1)\geq\pdots\geq\ell(\WW_\pp)\geq 2>1=\ell(\WW_{\pp+1})=\pdots=\ell(\WW_\qq),$$ we define
  $$\ff^{(2)}=\sum_{1\leq \tt\leq \pp}\alpha_\tt\WW_\tt \mbox{ and }
  \ff^{(1)}=\sum_{\pp+1\leq \tt\leq \qq}\alpha_\tt\WW_\tt.$$  
Now we can introduce the last main result of this paper, which offers an algorithm to decide whether an endomorphism~$\varphi$ of a finitely generated free metabelian Poisson algebra~$\MP\XX$ induced by the images of elements in~$\XX$ is an automorphism.

\begin{theorem}\label{auto}
  Let~$\XX=\{\bb_1\wdots\bb_\nn\}$ be a finite well-ordered set and let~$\MP\XX$ be the free metabelian Poisson algebra generated by~$\XX$. Suppose that~$\varphi$ is an endomorphism of~$\MP\XX$ induced by
  $$\varphi(\bb_\ii)=\ff_\ii=\ff_{\ii}^{(2)}+\ff_{\ii}^{(1)}, 1\leq \ii\leq \nn,$$
and suppose that~$\tl\varphi$ is the endomorphism of the subspace~$\kk\XX$ of~$\MP\XX$ spanned by~$\XX$ induced by~$\tl\varphi(\bb_\ii)=\ff_{\ii}^{(1)}$. Then~$\varphi$ is an isomorphism of~$\MP\XX$ if and only if $\tl\varphi$ is an isomorphism of~$\kk\XX$ and every~$\ff_{\mm}^{(2)}$ $(\mm\leq \nn)$ lies in the ideal~$\Id\SS$ of~$\MP\XX$ generated by the set~$\SS$ defined by~$$\SS=\{(\ff_\ii,\ff_\jj), \ff_\ii\cdot\ff_\jj \mid \ii,\jj\leq \nn, \varphi(\bb_\ii)=\ff_\ii, \varphi(\bb_\jj)=\ff_\jj\}.$$
In particular, there exists an algorithm to decide whether~$\varphi$ is an isomorphism of~$\MP\XX$.
\end{theorem}

\begin{proof}
Suppose that~$\varphi$ is an isomorphism of~$\MP\XX$, then clearly~$\tl\varphi$ is an isomorphism of~$\kk\XX$, and thus~$\ff_1^{(1)}\wdots\ff_{\nn}^{(1)}$ are linearly independent. Moreover, every~$\ff_{\mm}^{(1)}$ $(\mm\leq \nn)$ lies in the subalgebra of~$\MP\XX$ generated by~$\{\ff_\ii\mid 1\leq \ii\leq \nn,\varphi(\bb_\ii)=\ff_\ii\}$.
Assume
\begin{equation}\label{sum1}
\ff_{\mm}^{(1)}=\sum_{\ii}\alpha_{\mm,\ii}\ff_{\ii}+T(\ff_1\wdots\ff_\nn),
\end{equation}
where~$T$ is a linear combination of elements of length~$\geq2$ in~$\mpb$, and~$T(\ff_1\wdots\ff_\nn)$ is the resulting polynomial after the substitution of~$\bb_\ii$ by~$\ff_\ii$ for every~$\ii\leq \nn$ simultaneously. Summarizing all the monomials of length~1 in Eq.~\eqref{sum1}, we know~$\alpha_{\mm,\mm}=1$ and~$\alpha_{\mm,\ii}=0$ for every~$\ii\neq\mm$. Therefore, we deduce
$$
\ff_{\mm}^{(2)}=-T(\ff_1\wdots\ff_\nn)\subseteq \Id\SS.
$$

On the contrary, it is enough to show that~$\varphi$ is an epimorphism of~$\MP\XX$.  Assume
\begin{equation}\label{sum2}
\bb_\ii=\sum_{\jj}\beta_{\ii,\jj}\ff_{\jj}^{(1)}.
\end{equation}
Then for every~$\WW\in\mpb$ satisfying~$\ell(\WW)\geq 2$, we have
$$(\WW,\bb_\ii)=(\WW,\sum_{\jj}\beta_{\ii,\jj}\ff_{\jj}^{(1)})
=(\WW,\sum_{\jj}\beta_{\ii,\jj}\ff_{\jj})$$
and
$$\WW\cdot\bb_\ii=\WW\cdot(\sum_{\jj}\beta_{\ii,\jj}\ff_{\jj}^{(1)})
=\WW\cdot(\sum_{\jj}\beta_{\ii,\jj}\ff_{\jj}).$$ 
It follows that~$\Id\SS\subseteq \langle \ff_1\wdots\ff_\nn\rangle$, where~$\langle \ff_1\wdots\ff_\nn\rangle$ is the subalgebra of~$\MP\XX$ generated by~$\{\ff_1\wdots\ff_\nn\}$. In particular, we deduce~$\ff_{\mm}^{(2)}\in \Id\SS\subseteq \langle \ff_1\wdots\ff_\nn\rangle$ for every~$\mm\leq \nn$. Therefore,  we
have~$\ff_{\mm}^{(1)}\in\langle \ff_1\wdots\ff_\nn\rangle$ for every~$\mm\leq \nn$. And thus by Eq.~\eqref{sum2}, $\varphi$ is an epimorphism.

In particular, assume~$\ff_\ii^{(1)}=\sum_{\jj}\gamma_{\ii,\jj}\bb_\jj$. Then it is clear that~$\tl\varphi$ is an isomorphism of~$\kk\XX$ if and only if the determinant of the matrix~$(\gamma_{\ii,\jj})_{\nn\times\nn}$ is not zero. 
Finally, by Theorem~\ref{finite-GSB}, whether every~$\ff_{\mm}^{(2)}$ $(\mm\leq \nn)$ lies in the ideal~$\Id\SS$ is also decidable. The result follows.
\end{proof}
With the notations of Theorem~\ref{auto}, if one finds that~$\varphi$ is an isomorphism of~$\MP\XX$, then the next question is to ask whether~$\varphi$ is tame or not.
According to the introduction in~\cite{auto-leibniz}, we tend to believe that there are wild automorphism of~$\MP\XX$, because the subalgebra of~$\MP\XX$ generated by~$\aa\cdot\aa$ is clearly not a free metabelian Poisson algebra, where~$\aa$ is a letter in~$\XX$. 

We conclude the paper with the following problem: For a finitely generated free metabelian Poisson algebra over a computable field, does there exist an algorithm to decide whether an isomorphism is tame or not?

\newcommand{\noopsort}[1]{}

\end{document}